\documentclass[11pt]{article}
\usepackage{amsthm,amsmath,amssymb}
\usepackage{graphicx}
\usepackage{float}
\usepackage[font=small]{caption}
\numberwithin{figure}{section}

\newtheorem{theorem}{Theorem}
\numberwithin{theorem}{section}
\newtheorem{lemma}[theorem]{Lemma}

\newtheorem{corollary}[theorem]{Corollary}

\begin{document}

\title{Excluding a large theta graph}
\author{Guoli Ding\footnote{Research supported in part by NSF grant DMS-1500699} \hspace{.00001 in} and Emily Marshall \\ Mathematics Department, Louisiana State University, Baton Rouge, LA 70803}

\maketitle

\begin{abstract} A \textit{theta graph}, denoted $\theta_{a,b,c}$, is a graph of order $a+b+c-1$ consisting of a pair of vertices and three independent paths between them of lengths $a$, $b$, and $c$. We provide a complete characterization of graphs that do not contain a large $\theta_{a,b,c}$ as a topological minor. More specifically, we describe the structure of $\theta_{1,2,t}$-, $\theta_{2,2,t}$-, $\theta_{1,t,t}$-, $\theta_{2,t,t}$-, and $\theta_{t,t,t}$-free graphs where $t$ is large. The main result is a characterization of $\theta_{t,t,t}$-free graphs for large $t$. The $3$-connected $\theta_{t,t,t}$-free graphs are formed by $3$-summing graphs without a long path to certain planar graphs. The $2$-connected $\theta_{t,t,t}$-free graphs are then built up in a similar fashion by 2- and 3-sums. This result implies a well-known theorem of Robertson and Chakravarti on graphs that do not have a bond containing three specified edges.    
\end{abstract}

\section{Introduction}

All graphs are loopless but may have parallel edges. Undefined terminology can be found in \cite{diestel}. 

In this paper, we describe the structure of graphs that do not contain certain large theta graphs as a minor. A \textit{theta graph}, denoted $\theta_{a,b,c}$, is a graph of order $a+b+c-1$ consisting of a pair of vertices and three independent paths between them of lengths $a$, $b$, and $c$. Theta graphs have maximum degree 3 so containing a theta graph as a minor is equivalent to containing a theta graph as a topological minor. Throughout we will say \textit{$G$ contains $\theta_{a,b,c}$} to mean $G$ contains $\theta_{a,b,c}$ as a minor (or topological minor). Additionally we use the phrase \textit{$G$ contains a $\theta_{a,b,c}$ graph at $u$ and $v$} to mean $G$ contains as a subgraph a subdivision of $\theta_{a,b,c}$ in which $u$ and $v$ are the two vertices of degree 3. A graph is \textit{$\theta_{a,b,c}$-free} if it does not contain $\theta_{a,b,c}$.

The main goal of this paper is to describe all $\theta_{t,t,t}$-free graphs for large integers $t$. In other words, we want to characterize all graphs that do not contain three long independent paths between any pair of vertices. This problem is in fact an instance of a very general problem (P): for a given class $\cal H$ of graphs, determine all minor-closed classes $\cal G$ of graphs for which $\mathcal G\not\supseteq \cal H$. Our problem is exactly (P) when $\cal H$ is the class of all theta graphs. There are several choices of $\cal H$ for which (P) has been solved. Along this line, the best known results are the two obtained by Robertson and Seymour which solve (P) for the class of all complete graphs \cite{RoSe03} and for the class of all planar grids \cite{RoSe86}. The same authors also solved (P) for the classes of all trees, all stars, and all paths \cite{RoSe83, RoSe85}. Other classes for which (P) is solved include the class of all wheels \cite{DiDzWu} and the class of all double-paths \cite{ding}.

We prove that $\theta_{t,t,t}$-free graphs have the following structure: begin with a planar graph that contains no long paths outside of a special facial cycle and attach graphs that do not have long paths to the planar graph along edges, facial triangles, and certain facial $4$-cycles. This result is stated formally in the next section. Additionally, we describe all $\theta_{1,2,t}$-, $\theta_{2,2,t}$-, $\theta_{1,t,t}$-, and $\theta_{2,t,t}$-free graphs. 

Our result for $\theta_{t,t,t}$-free graphs implies a result of Robertson and Chakravarti \cite{RoCh80} concerning when three specified edges of a graph are contained in a \textit{bond} (a minimal nonempty edge-cut of the graph). Suppose we subdivide the three specified edges sufficiently many times. Then it is easy to see that the three specified edges are contained in a bond in the original graph if and only if the subdivided graph contains $\theta_{t,t,t}$. This connection easily leads to the result of Robertson and Chakravarti, as we will see in the next section, and it also illustrates how much our result strengthens their result.

Another important goal of this paper is to develop tools for dealing with various cases of problem (P). We will prove several key lemmas that could be used in similar situations. In particular, we will obtain a strengthened version of a result of Robertson and Seymour \cite{RoSe90} on the embeddability of a graph on a disc. 

The remainder of the paper is organized as follows. In Section 2 we formalize and state our main theorem. In Section 3 we examine graphs with a long path and look at large graphs which are necessarily present in such graphs. Section 4 describes several ways we will decompose our graphs into smaller pieces which will be useful in proofs. Section 5 includes lemmas on weighted graphs. In Section 6 we state and prove the characterizations of $\theta_{1,2,t}$-, $\theta_{2,2,t}$-, $\theta_{1,t,t}$-, and $\theta_{2,t,t}$-free graphs. Section 7 extends and strengthens a result of Robertson and Seymour concerning planar drawings of graphs and crossing paths. In Section 8 we prove our main theorem for $3$-connected graphs. Finally in Section 9 we complete our proof of the main theorem by considering $2$-connected graphs.

\section{Statement of the main theorem}

Let $G$ be a graph. For any two adjacent vertices $x$ and $y$, the set of all edges between $x$ and $y$ is called a {\it parallel family} of $G$. 
A {\it simplification} of $G$, denoted $si(G)$, is a simple graph obtained from $G$ by deleting all but one edge from each parallel family. We call $G$ \textit{$3$-connected} if $si(G)$ is 3-connected. We call $G$ {\it $2$-connected} if either $si(G)$ is 2-connected or $si(G)=K_2$ with $||G||\ge2$. Because $\theta_{a,b,c}$ is $2$-connected, a graph is $\theta_{a,b,c}$-free if and only if each of its blocks is $\theta_{a,b,c}$-free. Therefore, we only need to determine $2$-connected $\theta_{a,b,c}$-free graphs. 

For any subgraph $H$ of $G$, a path $P$ of $G$ is an {\it $H$-path} if $E(P \cap H) = \emptyset$ and the distinct ends of $P$ are the only two vertices of $P$ that are in $H$. Let $C$ be a facial cycle of a plane graph $G$. If $C$ bounds the infinite face of $G$ then $C$ is called the {\it outer cycle}; if $C$ bounds a finite face then $C$ is an {\it inner cycle}. Note that $C$ is both inner and outer if $G=C$. For any cycle $C$, we always assume there is an implicit forward direction. This is for the purpose of simplifying our terminology. For any two vertices $u,v$ of $C$, denote by $C[u,v]$ the forward path of $C$ from $u$ to $v$.  

In our proof, it becomes convenient to consider weighted graphs. This notion also allows us to obtain a stronger result. A \textit{weight function} of a graph $G$ is a mapping $w$ from $E(G)$ to the set of positive integers. A graph with a weight function is called a \textit{weighted graph} and is denoted $(G,w)$. For any subgraph $G'$ of $G$, the weight of $G'$, denoted $w(G')$, is the sum of $w(e)$ over all edges $e$ of $G'$. We say $(G,w)$ {\it contains} $\theta_{a,b,c}$ if $G$ contains a theta graph as a subgraph for which the three independent paths have weights at least $a$, $b$, and $c$, respectively. Naturally, $(G,w)$ is $\theta_{a,b,c}$-{\it free} if it does not contain $\theta_{a,b,c}$. Our main result in fact is a characterization of $\theta_{t,t,t}$-free weighted graphs. To describe these weighted graphs, we first define two fundamental classes of weighted graphs. 

Let $r,s \geq 2$ be integers. Let $\mathcal L_s$ be the class of $2$-connected graphs that do not contain a path of length $s$. Let $\mathcal L_{r,s}$ be the class of weighted graphs $(G,w)$ with $G \in \mathcal L_s$ and $w(e) < r $ for all $e \in G$. It is clear that weighted graphs in $\mathcal L_{r,s}$ do not contain $\theta_{t,t,t}$ if $t \geq rs$. 

For any integer $r\ge2$, let $\mathcal P_r$ be the class of $2$-connected weighted plane graphs $(G,w)$ such that if $C$ is the outer cycle then $G$ contains no $C$-path of weight $\ge 2r$ and $G\backslash E(C)$ contains no edge of weight $\ge r$. It is not difficult to see that weighted graphs in $\mathcal P_r$ contain no $\theta_{t,t,t}$ for sufficiently large $t$. We do not justify this observation here since a more general statement will be proved later.   

General $\theta_{t,t,t}$-free weighted graphs will be constructed from $\mathcal L_{r,s}$ and $\mathcal P_r$ by $k$-sums which are defined as follows for $k=2,3,4$. Let $G_1$ and $G_2$ be two disjoint graphs. A \textit{$2$-sum} of $G_1$ and $G_2$ is a new graph formed by identifying a specified edge of $G_1$ with a specified edge of $G_2$ and then deleting the edge after identification. Similarly, for $k=3,4$, a \textit{$k$-sum} of $G_1$ and $G_2$ is a new graph formed by identifying a specified $k$-cycle of $G_1$ with a specified $k$-cycle of $G_2$ and then deleting the edges of the $k$-cycle after identification. The specified edge or $k$-cycle of each $G_i$ will be called the {\it summing edge} or {\it summing $k$-cycle}, respectively. If $w_1,w_2$ are weight functions of $G_1,G_2$, then a {\it $k$-sum} ($k=2,3,4$) of $(G_1,w_1)$ and $(G_2,w_2)$ is a new weighted graph $(G,w)$ such that $G$ is a $k$-sum of $G_1,G_2$ and for each $e\in G$, $w(e)=w_i(e)$ where $i$ is such that $e \in G_i$.

Let $G$ be a plane graph and let $C$ be its outer cycle. An inner facial 4-cycle $R=x_1x_2x_3x_4x_1$ of $G$ is called a {\it rectangle} if the four vertices of $R$ are all on $C$ and the two edges $x_1x_2$ and $x_3x_4$ of $R$ are also edges of $C$. Note this implies there are no edges parallel to either $x_1x_2$ or $x_3x_4$. 

For any integers $r,s\ge2$, let $\Phi(\mathcal L_{r,s},\mathcal P_r)$ denote the class of 2-connected weighted graphs obtained from weighted graphs $(G_0,w_0) \in \mathcal P_r$ by $k$-summing $(k=2,3,4)$ weighted graphs from $\mathcal L_{r,s}$ to edges, inner facial triangles, and rectangles of $G_0$. 
We call $G_0$ the {\it base graph} of $G$. 
Now we are ready to state our main theorem. 

\begin{theorem}
There exists a function $t(r,s)$ such that all weighted graphs in $\Phi(\mathcal L_{r,s},\mathcal P_r)$ are $\theta_{t,t,t}$-free. 
Conversely, there also exists a function $s(t)$ such that every $2$-connected $\theta_{t,t,t}$-free weighted graph belongs to $\Phi(\mathcal L_{t,s(t)},\mathcal P_t)$
\label{thm:big}
\end{theorem}

Since every graph $G$ can be viewed as a weighted graph $(G,\varepsilon)$ where $\varepsilon(e)=1$ for all $e\in G$, Theorem \ref{thm:big} also characterizes graphs that are $\theta_{t,t,t}$-free. We do not formally state this simplified characterization since its derivation is straightforward and the final formulation is almost identical to Theorem \ref{thm:big}. 

In the following we formally state the result of Robertson and Chakravarti \cite{RoCh80} and we prove it using Theorem \ref{thm:big}. 

\begin{corollary} Let $G$ be a $2$-connected graph with three distinct edges $e,f,g$. Then either $G$ has a bond containing $e,f,g$ or $G$ is obtained from a $2$-connected plane graph $G_0$ 
by $2$- and $3$-summing graphs to edges and inner facial triangles of $G_0$, where $e,f,g$ are contained in three graphs that are 2-summed to three distinct edges of the outer cycle of $G_0$.
\label{cor:RobCha}
\end{corollary}

\begin{proof} Suppose $G$ does not have a bond containing $e,f,g$. Let $t=|G|$ and let $w$ be a weight function of $G$ with $w(e)=w(f)=w(g)= t$ and $w(x)=1$ for all other edges $x$ of $G$. Then $(G,w)$ is $\theta_{t,t,t}$-free. By Theorem \ref{thm:big}, $(G,w)$ is obtained by summing weighted graphs from $\mathcal L_{t,s(t)}$ to $(H_0,w_0)\in\mathcal P_t$. Let $C$ be the outer cycle of $H_0$. Since no member of $\mathcal L_{t,s(t)}$ has an edge of weight $\ge t$ and since no edge of $H_0\backslash E(C)$ has weight $\ge t$, it follows that $e,f,g$ are all contained in $C$. 
If no 4-sum is used in the construction of $G$ then $G_0=H_0$ satisfies the requirement. If 4-sum is used then $H_0$ admits a 2-separation that divides $C$ into two paths. In this case, by making the base graph smaller and by allowing the summing graphs to contain at most one of $e,f,g$, we can replace the 4-sum by a 2-sum in the construction of $G$. Therefore, 4-sum can be eliminated from the construction and thus the result follows immediately. 
\end{proof}


\section{Unavoidable large graphs}

Graphs without a sufficiently long path are necessarily $\theta_{t,t,t}$-free. Since graphs without a long path have already been characterized by Robertson and Seymour \cite{RoSe85}, we will restrict our focus to graphs that do have a long path. The presence of a long path in a graph often implies the presence of some other large structure as well. In this section, we prove several lemmas describing these large structures. 

We begin with two lemmas that describe the unavoidable large structures in connected graphs with many vertices and in trees with many leaves, respectively. These will be used in our later proofs. Denote by $\Delta(G)$ the maximum degree of a vertex in $G$.  

\begin{lemma} If $G$ is simple, connected, and of order exceeding $1+d+d(d-1)+...+d(d-1)^ {p-1}$, where $d,p\ge1$ are integers, then either $\Delta(G) > d$ or $G$ has an induced path of length $p + 1$ starting from any specified vertex.
\label{lem:4.2.2}
\end{lemma}

\begin{proof} Suppose $\Delta(G) \leq d$. Let $v \in V (G)$ and let $n_k$ be the number of vertices of distance $k$ away from $v$. Then $n_0 = 1$, $n_1 = d_G (v)$, and $n_k \leq n_{k-1} (d - 1)$ for all $k \geq 2$. It follows that $|G| > n_0 + n_1 + \dots + n_p$ and thus $n_{p+1} \neq 0$. Therefore, $G$ has a vertex of distance $p + 1$ away from $v$, which proves the lemma. 
\end{proof}

\begin{lemma} If $T$ is a tree with at least $d^t$ leaves, where $d,t \geq 2$ are integers, then either $\Delta(T) > d$ or $T$ contains a subdivision of $\text{comb}_t$, which is shown on the left of Figure~\ref{fig:ladder}.
\label{lem:4.2.7}
\end{lemma}

\begin{proof} Since contracting an edge incident with a degree 2 vertex does not change the problem, assume $T$ has no vertex of degree 2. Since $d^t \geq 4$, $T$ has a vertex $v$ of degree greater than 2. If $T$ has a path of length $t$ starting from $v$ (which is necessarily induced), then a $\text{comb}_t$ subgraph can be obtained by extending this path. Assume no such path exists. Since $T$ has at least $d^t$ leaves, it follows that $|T| > d^t > 1 + d + d^2 + \dots + d^{t-1} > 1+d+d(d-1)+ \dots +d(d-1)^{t-2}$. Thus we deduce from Lemma~\ref{lem:4.2.2} that $\Delta(T) > d$. \end{proof}

\begin{figure}[ht]
\centerline{\includegraphics[scale=0.4]{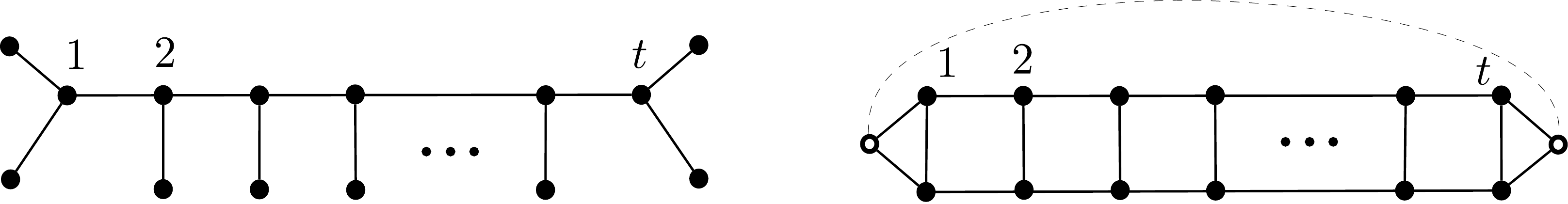}}
\caption{comb$_t$ and $L_t^+$} 
\label{fig:ladder}
\end{figure} 

Denote by $W_n$ the wheel on $n+1$ vertices and denote by $\ell(G)$ the length of a longest path in a graph $G$. The next result says that a 3-connected graph with a sufficiently long path must have a big wheel minor.

\begin{lemma} [\cite{DiDzWu}, Prop. 3.8] There exists a function $f_{\ref{lem:4.2.9}}(t)$ such that every $3$-connected graph $G$ with $\ell(G) \geq f_{\ref{lem:4.2.9}}(t)$ contains a $W_t$ minor. 
\label{lem:4.2.9}
\end{lemma}

Let $L_t$ be the graph shown on the right of Figure~\ref{fig:ladder} without the dashed edge and the white vertices.

\begin{lemma} Let $G$ consist of two disjoint paths $X = x_1x_2 \dots x_m$ and $Y = y_1y_2 \dots y_m$ and a matching $M = \{e_i = x_i y_{\pi(i)} : i = 1,2,\dots,m\}$. If $m > n^2$ then $G$ contains an $L_{n+1}$ (topological) minor. 
\label{lem:4.2.11}
\end{lemma}

\begin{proof} Let $e_i \prec e_j$ if $i < j$ and $\pi(i) < \pi(j)$. Let $F_1$ be the set of maximal members of $M$ with respect to $\prec$. Inductively, if $F_i$ has been defined and $M_i = M \backslash F_1 \backslash \dots \backslash F_i \neq \emptyset$, then let $F_{i+1}$ be the set of maximal members of $M_i$ with respect to $\prec$. Note members of each $F_i$ can be expressed as $e_{i_1} ,e_{i_2} ,\dots,e_{i_k}$ such that $i_1 < i_2 < \dots < i_k$ and $\pi(i_1) > \pi(i_2) > \dots > \pi(i_k)$. If $|F_i| > n$ for some $i$, then the conclusion holds since the union of paths $X,Y$ and matching $F_i$ contains $L_{n+1}$. Suppose $|F_i| \leq n$ for all $i$. Then $F_{n+1}\ne\emptyset$ since $m > n^2$. For each $i=2,...,n+1$ and each $e \in F_i$, note there exists $f \in F_{i-1}$ with $e \prec f$. Thus there exists $e_{i_j} \in F_j$ for $j = 1,2,\dots,n + 1$ such that $e_{i_{n+1}} \prec e_{i_n} \prec \dots \prec e_{i_1}$. Now the union of $X,Y$ and $e_{i_1},e_{i_2},\dots,e_{i_{n+1}}$ contains $L_{n+1}$. \end{proof}

Let $L_t^+$ be the graph shown on the right in Figure~\ref{fig:ladder} with the dashed edge. The next result strengthens Lemma \ref{lem:4.2.9}.

\begin{lemma} There exists a function $f_{\ref{lem:4.2.12}}(t)$ such that every $3$-connected graph $G$ with $\ell(G) \geq f_{\ref{lem:4.2.12}}(t)$ contains $W_t$ or $L_t^+$ as a topological minor. 
\label{lem:4.2.12}
\end{lemma}

\begin{proof} We will show $f_{\ref{lem:4.2.12}}(t)=f_{\ref{lem:4.2.9}}(s)$, where $s=(t-1)^r$ and $r={1+(t+1)^2}$, satisfies the theorem. Let $G$ be $3$-connected with $\ell(G)\ge f_{\ref{lem:4.2.12}}(t)$. By Lemma~\ref{lem:4.2.9}, $G$ has a $W_s$ minor. This minor can be considered as a cycle $C$ of length at least $s$ in $G$, a connected subgraph $G_0$ of $G$ with $V(G_0 \cap C) = \emptyset$, and a set $S$ of $s$ edges each incident with a vertex of $G_0$ and a distinct vertex of $C$. Let $G_1$ be the graph $G_0$ together with the edges in $S$. Let $T$ be a smallest tree of $G_1$ containing all edges of $S$. Then leaves of $T$ are precisely the $s$ vertices on $C$ that are incident with an edge of $S$. Now by Lemma~\ref{lem:4.2.7}, either $\Delta(T)> t-1$ or $T$ contains a subdivision of $\text{comb}_r$. First suppose the former and let $v$ be a vertex of degree at least $t$ in $T$. Then $T$ has $t$ independent paths from $v$ to leaves of $T$. Clearly, these paths together with $C$ form a  subdivision of $W_t$. 

Next suppose $T$ contains a subdivision $T'$ of $\text{comb}_r$. Let $X$ be the minimal path of $T'$ that contains all the $r$ cubic vertices of $T'$. Then $T$ contains a set $\cal P$ of $r$ disjoint paths from $X$ to $C$. Let $e$ be an edge of $C$ and let $Y=C\backslash e$. By viewing paths in $\cal P$ as a matching between $X$ and $Y$, we deduce from Lemma~\ref{lem:4.2.11} that the union of $X$, $Y$, and paths in $\cal P$ contains an $L_{t+2}$ topological minor. Now this topological minor together with $C$ contains an $L_t^+$ topological minor.         \end{proof}

\section{Decompositions}

It will be helpful in later proofs to decompose graphs into smaller pieces for the purpose of better understanding their structure. In this section, we describe several ways to do this. 

A {\it separation} of a graph $G$ is a pair $(G_1,G_2)$ of edge-disjoint non-spanning subgraphs of $G$ with $G_1\cup G_2=G$. A set $Z \subseteq V(G)$ is a \textit{cut} of $G$ if $G-Z$ is disconnected. It is clear that if $(G_1,G_2)$ is a separation then $V(G_1\cap G_2)$ is a cut. Conversely, if $Z$ is a cut then $G$ has a separation $(G_1,G_2)$ with $V(G_1 \cap G_2) = Z$. 
For any integer $k$, a \textit{$k$-separation} is a separation $(G_1,G_2)$ with $|V(G_1 \cap G_2)|=k$ and a \textit{$k$-cut} is a cut $Z$ with $|Z|=k$. The following lemma relates $k$-sum with $k$-separation. We omit the proof since it is easy.

\begin{lemma}\label{lem:sumsep}
$(a)$ Let $G$ be $2$-connected and let $(G_1,G_2)$ be a $2$-separation of $G$ with $V(G_1\cap G_2)=\{x,y\}$. For $i=1,2$, let $G_i^+$ be obtained from $G_i$ by adding a new edge $xy$. Then each $G_i^+$ is a $2$-connected minor of $G$ and $G$ is a $2$-sum of $G_1^+$ and $G_2^+$. \\ 
\indent $(b)$ Let $G$ be $3$-connected and let $(G_1,G_2)$ be a $3$-separation of $G$ with $V(G_1\cap G_2)=\{x,y,z\}$. For $i=1,2$, let $G_i^+$ be obtained from $G_i$ by adding three new edges $xy,yz,xz$. Then each $G_i^+$ is $3$-connected and $G$ is a $3$-sum of $G_1^+$ and $G_2^+$. Moreover, $G_i^+$ is a minor of $G$ unless $si(G_{3-i})=K_{1,3}$. \\ 
\indent $(c)$ Let $G$ be $k$-connected $(k=2,3)$ and be a $k$-sum of $G_1,G_2$, where $|G_1|,|G_2|>k$. For $i=1,2$, let $G_i'$ be obtained from $G_i$ by deleting its summing edge (when $k=2$) or its edges of the summing triangle (when $k=3$). Then $(G_1',G_2')$ is a $k$-separation of $G$.
\end{lemma}

For any disjoint graphs $G_0,G_1,\dots,G_k$ ($k\ge0$), let $S_2(G_0;G_1,\dots,G_k)$ denote a graph obtained by $2$-summing $G_i$ to $G_0$ for all $i >0$.

\begin{lemma} Let $e=xy$ be an edge of a $2$-connected graph $G$ of order at least three. Then $G$ has $2$-connected minors $G_0, G_1, ..., G_k$ such that $e \in G_0$, $|G_i|\ge3$ $(i>0)$, and $G=S_2(G_0;G_1,\dots,G_k)$. Moreover, if $\{x,y\}$ is a $2$-cut of $G$ then $si(G_0)=K_2$ and $k\ge2$; if $\{x,y\}$ is not a $2$-cut of $G$ then either $si(G_0)=K_3$ or $G_0$ is $3$-connected.
\label{lem:Sop}
\end{lemma}

\begin{proof} 

Suppose the result is false. Then we choose a counterexample $G$ with $|G|$ minimum. If $si(G)=K_3$ or $G$ is $3$-connected then the lemma holds with $k=0$; if $\{x,y\}$ is a 2-cut then the lemma also holds by Lemma \ref{lem:sumsep}(a). Thus $G$ has a $2$-separation but $\{x,y\}$ is not a 2-cut. It follows that $G$ can be expressed as a 2-sum of two 2-connected minors $G',G''$ over edges $e'$ of $G'$ and $e''$ of $G''$. Among all possible choices, let us choose $G',G''$ such that $|G'|$ is minimum with the property that $e\in G'$. Note $e$ and $e'$ are not parallel since $\{x',y'\}$ is a 2-cut of $G$, where $e'=x'y'$, but $\{x,y\}$ is not. By the minimality of $G$, $G'$ has 2-connected minors $G_0,G_1,...,G_k$ of order $\ge 3$ such that $e\in G_0$, $G'=S_2(G_0;G_1,...,G_k)$, and either $si(G_0)=K_3$ or $G_0$ is $3$-connected. Now by the minimality of $G'$ we also have $e'\in G_0$. Therefore, $G=S_2(G_0;G'',G_1,...,G_k)$, contradicting the choice of $G$, which proves the lemma.
\end{proof}

We also have a 3-connected version of the last lemma. For any disjoint graphs $G_0,G_1,\dots,G_k$ ($k\ge0$), let $S_3(G_0;G_1,\dots,G_k)$ denote a graph obtained by $3$-summing $G_i$ to $G_0$ for all $i >0$. Let $G$ be 3-connected and let $Z \subseteq V(G)$. We call $(G,Z)$ \textit{$4$-connected} if for every $s$-separation $(G_1,G_2)$ of $G$ with $Z \subseteq V(G_1)$, either $s \geq 4$ or $s=3=|G_2|-1$. 

\begin{lemma} Let $G$ be $3$-connected and let $Z \subseteq V(G)$. If $Z$ is not a subset of any $3$-cut, then $G$ has a $3$-connected minor $G_0$ such that $Z \subseteq V(G_0)$, $(G_0,Z)$ is $4$-connected, and $G=S_3(G_0;G_1,\dots,G_k)$, where $G_1,...,G_k$ are $3$-connected of order $\ge5$. In addition, each $G_i$ $(i>0)$ is a minor of $G$ unless $si(G)$ has a cubic vertex $z$ such that $z$ is not in any triangle and $Z\subseteq \{z\}\cup N_G(z)$.
\label{lem:3sum}
\end{lemma}

\begin{proof} 
Suppose the result is false. Then we choose a counterexample $G$ with $|G|$ minimum. 
Since the result holds if $(G,Z)$ is 4-connected, we deduce $G$ has a 3-separation $(H_1,H_2)$ with $Z\subseteq V(H_1)$ and $|H_2|\ge5$. By Lemma \ref{lem:sumsep}(b), $G$ can be expressed as a 3-sum of two 3-connected graphs $G',G''$ such that $Z\subseteq V(G')$ and $|G''|\ge5$. Among all possible choices, let us choose $G',G''$ with $|G'|$ minimum. Note $G'$ is a minor of $G$ since $|G''|\ge5$. Also note $|G'|\ge5$ because otherwise $|G'|=4$ and trivially $(G',Z)$ is $4$-connected so $(G_0,G_1)=(G', G'')$ would satisfy the lemma, which contradicts the choice of $G$. As a result, $G''$ is also a minor of $G$. By the minimality of $G$, $G'$ has a 3-connected minor $G_0$ such that $Z\subseteq V(G_0)$, $(G_0,Z)$ is 4-connected, and $G'=S_3(G_0;G_1,\dots,G_k)$, where $|G_i|\ge5$ ($i>0$). By the minimality of $G'$, the summing triangle between $G'$ and $G''$ must be contained in $G_0$. From this triangle it follows that $G_1,...,G_k$ are all minors of $G'$ and $G=S_3(G_0;G'',G_1,\dots,G_k)$. This contradicts the choice of $G$ and thus the lemma is proved. 
\end{proof} 

The previous two lemmas are about how a graph can be decomposed into a star structure with a better connected center. In the following we consider how to decompose a graph into a path structure. Let $e=x_0y_0$ be a specified edge of a $2$-connected graph $G$. A sequence $G_0,G_1,...,G_n$ ($n\ge0$) of edge-disjoint subgraphs of $G$ is called a \textit{chain decomposition} of $G$ at $e$ with {\it length} $n$ if \\ 
\indent (i) $e\in G_0$; \\ 
\indent (ii) for each $i=1,...,n$, $(G_0\cup ... \cup G_{i-1},\ G_i\cup ... \cup G_n)$ is a 2-separation of $G$; \\ 
\indent (iii) let $\{x_i,y_i\} = V((G_0\cup ... \cup G_{i-1})\cap (G_i\cup ... \cup G_n))$ for $i=1,...,n$; then 
the pairs $\{x_0,y_0\}$, $\{x_1,y_1\}$, ..., $\{x_n,y_n\}$ are all distinct. 

\begin{figure}[ht]
\centerline{\includegraphics[scale=0.5]{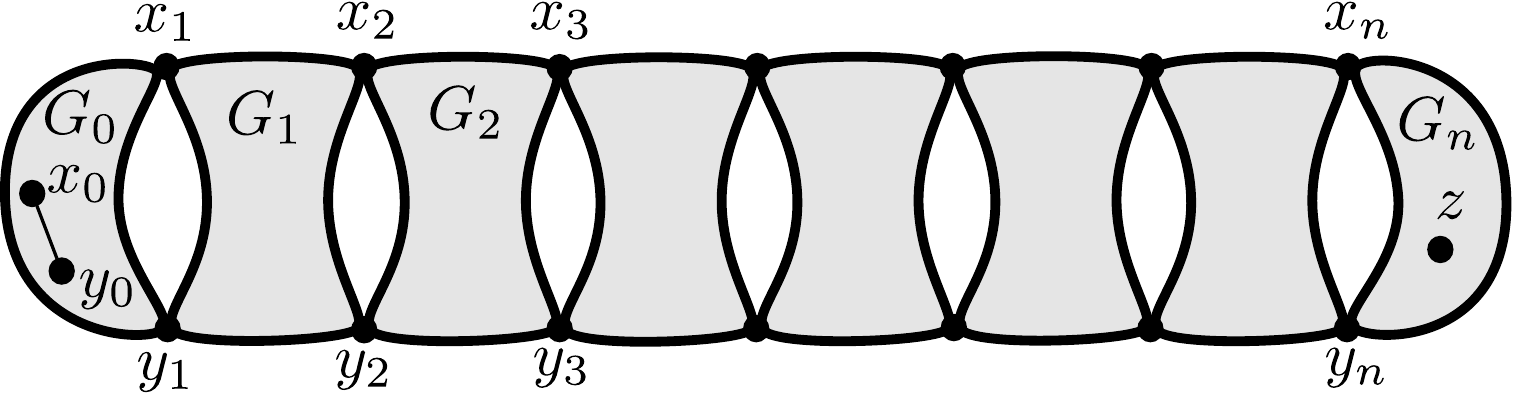}}
\caption{a chain decomposition\label{fig:bpath}}
\end{figure} 

We point out that $\{x_i,y_i\} \cap \{x_{i+1},y_{i+1}\} \neq \emptyset$ is allowed. It is clear that every 2-connected $G$ admits a chain decomposition of length 0 at any of its edges since conditions (ii-iii) are trivially satisfied. Let $a(G,e)$ denote the largest length of a chain decomposition of $G$ at $e$.

Chain decompositions and ``star" decompositions are similar, yet each allow us to focus on different aspects of a graph. A star decomposition focuses on how a graph is built around one central piece and will be used later in the paper when we have a known subdivision in a graph and want to look at possible extensions of the subdivision. A chain decomposition looks at how a graph can be broken down into a chain of $2$-connected pieces and is useful in determining long paths in a graph. The next lemma involves both decompositions.

By \textit{operation} $S$ we mean the operation of constructing $S_2(G_0;G_1,...,G_k)$ from $G_0,...,G_k$.  Starting from any class of graphs we may construct more graphs by applying operation $S$ repeatedly. In the following we make this more precise. Let $\cal G$ be a class of graphs. 
Let $\mathcal G_0$ be the class of all pairs $(G,e)$ such that $G\in\cal G$ and $e$ is an edge of $G$. For any positive integer $n$, if $\mathcal G_{n-1}$ has been  defined, let $\mathcal G_n$ consist of all pairs $(G,e)$ for which there exist $(G_0,e)\in \mathcal G_0$ and $(G_i,e_i)\in \mathcal G_{n-1}$ ($i=1,...,k$) such that $G$ is obtained by 2-summing $G_i$ to $G_0$ over $e_i$ for all $i>0$. We say each $(G,e)\in \mathcal G_n$ is constructed from graphs in $\mathcal G$ by $n$ \textit{iterations} of operation $S$.

\begin{lemma} Let $e$ be a specified edge of a $2$-connected graph $G$ with $a(G,e) \le a$. Then $(G,e)$ can be constructed from its $3$-connected minors and 2-connected minors of order $2$ or $3$ by at most $a+1$ iterations of operation $S$. 
\label{lem:aGebound}
\end{lemma}

\begin{proof} 
Let $x,y$ be the two ends of $e$. We first assume $|G|>2$ and $\{x,y\}$ is not a 2-cut. In this case we claim $(G,e)$ can be constructed from its $3$-connected minors and 2-connected minors of order 3 within $a$ iterations. Suppose the claim is false. Choose a counterexample with $|G|$ as small as possible. By Lemma \ref{lem:Sop}, $G$ has 2-connected minors $G_0,G_1,...,G_k$ such that $e\in G_0$, either $si(G_0)=K_3$ or $G_0$ is 3-connected, $|G_i|\ge3$ ($i>0$), and $G=S_2(G_0;G_1,...,G_k)$. For each $i>0$, let $e_i=x_iy_i$ be the summing edge of $G_i$. By allowing different graphs to sum over edges of $G_0$ from the same parallel family, we may assume $G_i-\{x_i,y_i\}$ is connected. Then $a(G_i,e_i)\le a-1$ because otherwise, since $G-\{x,y\}$ is connected, we would have $a(G,e)\ge a(G_0\cup G_i, e) >a$. By the minimality of $G$, we deduce that each $(G_i,e_i)$ can be constructed from its $3$-connected minors and 2-connected minors of order 3 within $a-1$ iterations. It follows that $(G,e)$ can be constructed from its $3$-connected minors and 2-connected minors of order 3 within $a$ iterations. This conclusion contradicts the choice of $G$ and thus proves our claim. 

If $|G|=2$ then $a(G,e)=0$ and it is clear that $(G,e)$ can be constructed in at most one iteration. 
Now suppose $G-\{x,y\}$ has $k>1$ components. Let $G_0$ consist of $e$ and $k$ other edges parallel with $e$. Then $G$ has 2-connected minors $G_1,...,G_k$ of order $\ge3$ such that $G=S_2(G_0;G_1,...,G_k)$. For each $i=1,...,k$, let $G_i$ be summed to $G_0$ over $e_i$. Note $G_i-\{x,y\}$ is connected and $a(G_i,e_i)\le a$ for every $i$. By the above claim, every $(G_i,e_i)$ can be constructed within $a$ iterations, which implies $(G,e)$ can be constructed within $a+1$ iterations. 
\end{proof}

\section{Weighted graphs} 

In this section we prove a few technical lemmas on weighted graphs.

\begin{lemma} Let $t\ge2$ be an integer and let $(G,w)$ be a $2$-connected weighted graph with a path of weight  exceeding $(t-2)^2$. Then $G$ has a cycle of weight at least $t$ and, for any two distinct vertices $u,v$, a $uv$-path of weight at least $t/2$. \label{lem:lemmaa}\end{lemma}

\begin{proof} Let $P=x\dots y$ be a path of $G$ of weight at least $(t-2)^2+1$. We first show $G$ has a cycle $C$ of weight at least $t$. Let $C'$ be a cycle containing $x$ and $y$. We assume $w(C')<t$ because otherwise $C=C'$ satisfies the requirement. Then $V(P \cap C')$ divides $P$ into at most $t-2$ subpaths and hence at least one subpath $P'$ must have weight at least $t-1$. Clearly, $P' \cup C'$ contains a cycle $C$ of weight at least $t$, as required. Since $G$ is $2$-connected, for every distinct pair of vertices $u,v$, there exist disjoint paths between $\{u,v\}$ and $C$ (where $u,v$ may be on $C$). These two paths together with $C$ contain a $uv$-path of weight at least $t/2$. \end{proof}

\begin{lemma} Let $(G, w)$ be a $2$-connected weighted graph of order $\ge3$ and let $t$ be a positive integer. Then one of the following holds. \\ 
\indent (a) $G$ has a $2$-separation $(H, J)$ with $V(H\cap J)=\{x,y\}$ such that neither $H$ nor $J$ has an $xy$-path of weight $\ge t$. \\ 
\indent (b) $G$ has a $2$-separation $(H, J)$ with $V(H\cap J)=\{x,y\}$ such that both $H$ and $J$ have an $xy$-path of weight $\ge t$. \\ 
\indent (c) $G=S_2(G_0;G_1,...,G_k)$ such that either $si(G_0)=K_3$ or $G_0$ is $3$-connected, and for each $i>0$, if $e_i=x_iy_i$ is the summing edge of $G_i$ then $G_i\backslash e_i$ has no $x_iy_i$-path of weight $\ge t$.
\label{lem:2sep}
\end{lemma}

\begin{proof} 
Suppose the lemma is false. Let $(G,w)$ be a counterexample on the fewest vertices. If $G$ has no 2-separations then (c) would hold with $k=0$. Hence $G$ has a $2$-separation $(H, J)$ with $V(H\cap J)=\{x,y\}$ such that $H$ has an $xy$-path of weight $\ge t$ but $J$ does not. Among all such 2-separations we choose one with $|H|$ minimum. Since $|H|$ is a minimum, if $H-\{x,y\}$ is not connected, then (b) would hold; thus $H-\{x,y\}$ is connected. Let $H^+$ be formed from $H$ by adding a new edge $e_H=xy$ and let $J^+$ be formed similarly. By Lemma \ref{lem:Sop}, $H^+$ has 2-connected minors $G_0,G_1,...,G_k$ of order $\ge 3$ such that $e_H\in G_0$, either $si(G_0)=K_3$ or $G_0$ is 3-connected, and $H^+=S_2(G_0;G_1,...,G_k)$. For each $i>0$, let $G_i$ be 2-summed to $G_0$ over $e_i=x_iy_i$. Then the minimality of $H$ implies $G_i\backslash e_i$ has no $x_iy_i$-path of weight $\ge t$. It follows that $G=S_2(G_0;J^+,G_1,...,G_k)$ and the decomposition satisfies (c). This contradicts the choice of $G$ and thus it proves the lemma. \end{proof}      

In the next lemma we use the following terminology. Let $(G,w)$ be a weighted graph and let $(G_1,G_2)$ be a 2-separation of $G$ with $V(G_1 \cap G_2) = \{x,y\}$. For $i=1,2$, define $(G_i^+,w_i)$ where $G_i^+$ is obtained from $G_i$ by adding a new edge $e_i=xy$, $w_i(e_i)$ is equal to the maximum weight of an $xy$-path in $G_{3-i}$, and $w_i(e) = w(e)$ for all edges $e$ of $G_i$. 

\begin{lemma}
$(G,w)$ contains $\theta_{a,b,c}$ {\small if and only if} at least one of $(G_1^+,w_1)$ and $(G_2^+,w_2)$ contains $\theta_{a,b,c}$.
\label{lem:lemmab}
\end{lemma}

\begin{proof}
Suppose $(G,w)$ contains $\theta_{a,b,c}$. Then $G$ contains two vertices $u,v$ and three independent $uv$-paths $P_1,P_2,P_3$ of weight at least $a,b,c$, respectively. Observe that both $u,v$ are contained in $G_i$ for some $i$ because otherwise we would have $u\in G_j-\{x,y\}$ and $v\in G_{3-j}-\{x,y\}$ for some $j$, which is impossible. Let $T=P_1\cup P_2\cup P_3$. Then either $T \subseteq G_i$ or $T\cap G_{3-i}\subseteq P_j$ for some $j$. In the first case $T$ is a $\theta_{a,b,c}$ contained in $(G_i^+,w_i)$ while in the second case replacing $T\cap G_{3-i}$ by $e_i$ in $T$ results in a $\theta_{a,b,c}$ contained in $(G_i^+,w_i)$.

Conversely, suppose some $(G_i^+,w_i)$ contains a $\theta_{a,b,c}$ graph $T$. If $e_i\notin T$ then $T$ is a $\theta_{a,b,c}$ graph of $(G,w)$. So assume $e_i\in T$. Form a new theta graph $T'$ by replacing $e_i$ in $T$ with an $xy$-path of $G_{3-i}$ of weight equal to the weight of $e_i$. Then $T'$ is a $\theta_{a,b,c}$ graph of $(G,w)$.
\end{proof}

\begin{lemma} Let $(G,w)$ be $\theta_{t,t,t}$-free, where $G$ is $3$-connected and planar. Suppose $C$ is a facial cycle such that $|C| \geq 3t$ or $C$ contains two edges each of weight $\ge t$. If each edge of $C$ has the maximum weight among edges parallel to it, then $G$ has no $C$-path of weight $\geq 2t$ and $G \backslash E(C)$ has no edge of weight $\geq t$. 
\label{lem:removeC}
\end{lemma}

\begin{proof} Suppose, for the sake of contradiction, $G$ contains a $C$-path $P$ with $w(P) \geq 2t$. Let $v_1$ and $v_2$ be the two ends of $P$. If $C[v_1,v_2]$ and $C[v_2,v_1]$ both have weight at least $t$, then there is a $\theta_{t,t,t}$ in $G$ at $v_1$ and $v_2$. Hence one of these paths, say $C[v_1,v_2]$, has weight less than $t$ and so $C[v_2,v_1]$ has a vertex $x$ such that $C[v_2,x]$ and $C[x,v_1]$ each has weight at least $t$.  
Note $|P|\ge3$ because otherwise, since $C$ is a facial cycle and $V(P)$ is not a 2-cut, $C[v_1,v_2]$ must have only one edge and this edge is parallel to the unique edge of $P$. This contradicts our assumption on $C$ since $w(P)>w(C[v_1,v_2])$. Since $G$ is $3$-connected, it has three independent paths $Q_1,Q_2,Q_3$ from $x$ to distinct vertices of $P$, where the paths are listed in the order in which their ends appear on $P$. If $C'$ is the cycle contained in $Q_1\cup Q_3\cup P$, then $Q_2$ intersects $C'$ only at $x$ and $P$, which implies $Q_2$ intersects $C$ only at $x$. Therefore, $C\cup P\cup Q_2$ contains a $\theta_{t,t,t}$ at $x$ and either $v_1$ or $v_2$.

Suppose $G \backslash E(C)$ has an edge $e$ with $w(e) \geq t$. Find two disjoint paths from the ends of $e$ to $C$ and let $P$ be the $C$-path consisting of $e$ and these two paths. Now by an argument similar to the one used above, we find a vertex $x$ and a path $Q_2$ from $x$ to $y$ on $P$ and then a $\theta_{t,t,t}$ in $(G,w)$. Previously, we required $w(P)\ge 2t$ so that at least one of the two subpaths of $P$ divided by $y$ would have length at least $t$.  Now since $P$ in this case contains an edge $e$ of weight at least $t$, taking the part of $P$ that contains $e$ will have the same result. \end{proof}

In the next lemma, the graphs in the statement are not weighted but a weighted graph is defined and used in the proof. 

\begin{lemma} If $k\ge1$ then $\ell(S_2(G_0;G_1,\dots,G_k)) \leq (\ell(G_0)+2) \cdot \text{max}\{\ell(G_1),\dots,\ell(G_k)\}$.
\label{lem:Sdecompbd} 
\end{lemma}

\begin{proof} For each $i=1,...,k$, let $e_i$ be the edge of $G_0$ such that $G_i$ is 2-summed to $G_0$ over $e_i$. Let $L=\max\{\ell(G_1),\dots,\ell(G_k)\}$. Let $w$ be a weight function of $G_0$ such that $w(e_i)=L$ for $i=1,...,k$ and $w(e)=1$ for all other edges. Now we consider any longest path $P$ of $S_2(G_0;G_1,\dots,G_k)$. Let $\cal Q$ be the set of all maximal subpaths $Q$ of $P$ such that $\emptyset\ne E(Q) \subseteq E(G_i)$ for some $i\ne0$. We modify $P$ as follows. For each $Q\in \cal Q$, if $Q$ is contained in $G_i$ and the two ends of $Q$ are the two ends of $e_i$ then in $P$ we replace $Q$ by $e_i$. Let $P'$ be the resulting path. Note $P'$ is the union of a path $P''$ of $G_0$ and up to two members of $\cal Q$, each containing an end of $P$. It follows that $||P||\le L + w(P'') + L \le \ell(G_0)L + 2L = (\ell(G_0)+2)L$. \end{proof}

\section{Excluding a large restricted theta graph}

In this section we prove characterizations of $\theta_{1,2,t}$-, $\theta_{2,2,t}$-, $\theta_{1,t,t}$-, and $\theta_{2,t,t}$-free graphs. For a proper subgraph $H$ of $G$, a \textit{bridge} of $H$ or an $H$-{\it bridge} is either a subgraph of $G$ induced by the edges of a component $C$ of $G-V(H)$ together with the edges linking $C$ to $H$, or a subgraph induced by an edge not in $H$ but with both ends in $H$. We will call the second type of bridges {\it trivial}. The vertices of an $H$-bridge that are in $H$ are the \textit{feet} of the bridge.

We begin with $\theta_{1,2,t}$-free graphs. The characterization is intuitive and requires only a very short proof. It is easy to see that cycles are $\theta_{1,2,t}$-free for all $t \geq 2$.  

\begin{theorem} Let $t \geq 2$ be an integer. Then every $2$-connected simple graph $G$ with $\ell(G) \geq 4t^2$ either contains $\theta_{1,2,t}$ or is a cycle. 
\label{thm:12t}
\end{theorem}

\begin{proof} Let $G$ be a $2$-connected simple graph with $\ell(G) \geq 4t^2$. By Lemma \ref{lem:lemmaa}, $G$ contains a cycle $C$ of length exceeding $2t$. If $G \neq C$, then $G$ has a bridge $B$ of $C$. Since $G$ is $2$-connected, $B$ has at least two feet along $C$, say $u$ and $v$. Suppose without loss of generality, $|C[u,v]| \geq |C[v,u]|$. Then $||C[u,v]||>t$ since $|C|>2t$. Let $Q$ be a $uv$-path of $B$. Then $C[u,v] \cup C[v,u]\cup Q$ is a subdivision of $\theta_{1,2,t}$ since $G$ is simple. \end{proof}  
 
A graph is \textit{outerplanar} if it has a plane embedding in which all vertices are on the outer cycle. Outerplanar graphs are known to be $\theta_{2,2,2}$-free ($\theta_{2,2,2} \cong K_{2,3}$) and thus are $\theta_{2,2,t}$-free for all $t \geq 2$. 

\begin{theorem} Let $t \geq 2$ be an integer. Then every $2$-connected graph $G$ with $\ell(G) \geq 4t^2$ either contains $\theta_{2,2,t}$ or is outerplanar.
\label{thm:22t}
\end{theorem} 

\begin{proof} Let $G$ be $2$-connected with $\ell(G) \geq 4t^2$ and let $C$ be a longest cycle of $G$. By Lemma \ref{lem:lemmaa}, $|C|>2t$. Suppose $C$ is not a Hamilton cycle. Then $G$ has a nontrivial bridge $B$ of $C$. Since $G$ is $2$-connected, $B$ has at least two feet along $C$, say $u$ and $v$. If $u$ and $v$ are adjacent along $C$, then $G$ contains a cycle longer than $C$: replace the edge $uv$ in $C$ with a path through $B$ of length $\ge 2$. Hence $u$ and $v$ are not adjacent in $C$ and thus $G$ contains a $\theta_{2,2,t}$ graph at $u$ and $v$: one path of length $\ge2$ is through $B$ and the other two paths are $C[u,v]$ and $C[v,u]$. Since $|C|\ge 2t$, one of these paths necessarily has length $\ge t$.

Now $C$ is a Hamilton cycle. Suppose $uv$, $xy$ are chords of $C$ such that $u,x,v,y$ are distinct and they appear in that forward order along $C$.  Since $|C|=|G|\ge \ell(G)\ge 4t^2$, at least one of $C[u,x], C[x,v], C[v,y], C[y,u]$ has length $\ge t$. Without loss of generality, suppose $||C[u,x]||\ge t$. Then $G$ contains a $\theta_{2,2,t}$ graph at $u$ and $x$: the path of length $\ge t$ is $C[u,x]$ and the two paths of length $\ge2$ each use one of the edges $uv$ and $xy$. Hence $C$ has no crossing chords and $G$ is outerplanar. \end{proof}

To describe $\theta_{1,t,t}$-free graphs, we define a new class of graphs. For any family $\mathcal{G}$ of $2$-connected graphs, let $C(\mathcal{G})$ be the class of graphs constructed by $2$-summing graphs from $\mathcal{G}$ to a cycle.

\begin{theorem} There exists a function $f_{\ref{thm:1tt}}(t)$ such that every $2$-connected graph $G$ with $\ell(G) \geq f_{\ref{thm:1tt}}(t)$ either contains $\theta_{1,t,t}$ or is in $C(\mathcal{L}_{8t^2})$ where $t \geq 3$ is an integer. Additionally, all graphs in $C(\mathcal{L}_t)$ are $\theta_{1,t,t}$-free. 
\label{thm:1tt}
\end{theorem}

\begin{proof} Let $w(t)=f_{\ref{lem:4.2.9}}(2t)$. We show that $f_{\ref{thm:1tt}}(t)=[w(t)+2]^{3t}w(t)$ satisfies the theorem. 
Suppose $\ell(G) \geq f_{\ref{thm:1tt}}(t)$ and further assume $G \notin C(\mathcal{L}_{8t^2})$. We need to show that $G$ contains $\theta_{1,t,t}$.  

Let $b=\text{max}\{\ell(G'): G' \text{ is a 3-connected minor of } G\}$; we know $b < w(t)$ since otherwise, by Lemma \ref{lem:4.2.9}, $G$ contains a $W_{2t}$ minor and hence a $\theta_{1,t,t}$. Let $e$ be a specified edge of $G$ and consider a chain decomposition of $G$ given by $G_0,G_1,\dots,G_n$ and with vertices $x_0,x_1,\dots, x_n, y_0, y_1, \dots, y_n, z$, as in Figure~\ref{fig:bpath}. If $a(G,e) < 3t$, then by Lemma \ref{lem:aGebound}, $G$ can be constructed from $3$-connected minors and graphs of order $\le3$ by at most $3t$ iterations of operation $S$. By Lemma~\ref{lem:Sdecompbd}, $\ell(G) \leq (b+2)^{3t} b < [w(t)+2]^{3t}w(t)$ which is a contradiction. 

Hence assume $a(G,e)=n \geq 3t$. Since $G$ is 2-connected, it has two independent paths from $z$ to $x_0$ and $y_0$. Without loss of generality, we assume one contains every $x_i$ and the other contains every $y_i$. 
Suppose $G_{t}\cup G_{t+1} \cup \dots \cup G_{n-t}$ contains a path $P$ from some $x_i$ to some $y_j$ and without loss of generality, assume $P$ does not include any other $x_k$ or $y_k$. Then there is a $\theta_{1,t,t}$ in $G$ at $x_i$ and $y_j$: the path of length $\ge1$ is $P$, one path of length $\ge t$ includes the vertices $x_{i-1},x_{i-2},\dots,x_0,y_0,y_1,\dots,y_{j-1}$ and the other includes the vertices $x_{i+1},x_{i+2},\dots,x_n,z,y_n,y_{n-1},\dots,y_{j+1}$. Hence no such path $P$ exists. It follows that each $G_i$ ($t\le i\le n-t$) has two components $G_i',G_i''$ such that $G_i'$ contains both $x_i,x_{i+1}$ and $G_i''$ contains both $y_i,y_{i+1}$. Therefore, $G$ can be constructed by 2-summing 2-connected graphs $H_1,...,H_k$ to a cycle $H_0$ of length $k>t$.  We choose these graphs with $k$ maximum.

Because $G \notin C(\mathcal{L}_{8t^2})$, $\ell(H_i) \ge 8t^2$ for some $i$. Let $xy$ be the summing edge of $H_i$. By the maximality of $k$, $H_i\backslash xy$ is 2-connected. Clearly, $\ell(H_i\backslash xy)\ge 4t^2$. Thus by Lemma \ref{lem:lemmaa}, $H_i\backslash xy$ has a cycle $C$ of length exceeding $2t$. Since $H_i$ is $2$-connected, it has disjoint paths from $x$ to a vertex $x'$ of $C$ and from $y$ to a vertex $y'$ of $C$ (where possibly $x=x'$ or $y=y'$). Now the 2-sum of $H_i$ and $H_0$ contains a $\theta_{1,t,t}$ graph at $x'$ and $y'$: $C[x',y']$ and $C[y',x']$ are paths of length $\geq t$ and $\ge1$, and the other path of length $\ge t$ is the union of the $xx'$-path, the $yy'$-path, and $H_0\backslash xy$. Consequently, $G$ contains $\theta_{1,t,t}$.

Finally, let $G \in C(\mathcal{L}_t)$. Suppose $G$ is formed by $2$-summing graphs $G_1,\dots,G_k$ to a cycle $C$. Suppose $G$ has a $\theta_{1,t,t}$ graph at $x$ and $y$. If $x \in V(G_i) \backslash V(C)$ for some $i$ then $y$ must also be in $V(G_i)$ because otherwise there could not be three independent paths from $x$ to $y$ since $G_i$ is separated from the rest of the graph by two vertices. But now at least one of the paths of length $\ge t$ would have to remain in $G_i$ which cannot happen since $\ell(G_i) <t$. Hence $x$ and $y$ must both be vertices of $C$. Because no $G_i$ has a path of length $\ge t$, the two paths of length $\ge t$ in any $\theta_{1,t,t}$ must each have an interior vertex in $C$. But now, no matter how these vertices are oriented with respect to $x$ and $y$ along $C$, there cannot be a $\theta_{1,t,t}$. \end{proof}  

The proof of the characterization of $\theta_{2,t,t}$-free graphs requires the following lemma. Let $G$ be a graph and let $e,f\in E(G)$. A subgraph $H$ of $G$ is called an $ef$-{\it theta} if $H$ is a theta graph such that, if $u,v$ are its two cubic vertices then $e,f$ belong to different $uv$-paths of $H$ and the third $uv$-path of $H$ has length $\ge2$. 
Suppose $e=xy$, $f=uv$, and $Z\subseteq V(G)$. Then we say $Z$ {\it separates} $e$ from $f$ if $\{x,y\} \backslash Z\ne\emptyset$, $ \{u,v\} \backslash Z \neq \emptyset$, and $G-Z$ has no path between $\{x,y\} \backslash Z$ and $\{u,v\} \backslash Z$. 

\begin{lemma} Let $e,f$ be distinct edges of a $2$-connected simple graph $G$. Suppose no two vertices of $G$ separate $e$ from $f$. Then $G$ contains an $ef$-theta unless either $e,f$ have a common end $v$ with $\text{deg}_G(v)=2$ or $e,f$ have no common end and $G=K_4$. 
\label{lem:eftheta}
\end{lemma}  

\begin{proof} Let $e=ab$ and $f=cd$. First consider the case $a=c$. Suppose $\text{deg}(a) \geq 3$ and let $x\in N_G(a) \backslash \{b,d\}$. Since $\{a,x\}$ does not separate $e,f$, there is a path $P$ from $b$ to $d$ in $G-\{a,x\}$. Furthermore, since $G$ is $2$-connected, there is a path $Q$ from $x$ to $P$ in $G-a$. Then the union of $P,Q,e,f$, and $ax$ is an $ef$-theta, as required.

Now $e,f$ is a matching. Assume $G$ does not contain an $ef$-theta. We will show $G=K_4$. Because $G$ is $2$-connected, it has a cycle $C$ containing $e,f$. Let $P,Q$ be the two paths of $C\backslash\{e,f\}$. Without loss of generality, assume $P$ is between $a$ and $c$ and $Q$ is between $b$ and $d$. Since $\{b,c\}$ does not separate $e,f$, there is an edge $pq$ with $p \in P-c$ and $q \in Q-b$. Choose such an edge $pq$ with $p$ as close to $a$ as possible along $P$. Since $\{p,q\}$ does not separate $e,f$, there is a path $R$ in $G-\{p,q\}$ between the two components of $C-\{p,q\}$. If the ends of $R$ are both on $P$ or both on $Q$, then the union of $R,C$, and $pq$ contains an $ef$-theta. So one end $p'$ of $R$ is on $P$ and the other end $q'$ of $R$ is on $Q$. It follows that $R$ has only one edge $p'q'$, $p'$ is between $p$ and $c$ along $P$ (by the choice of $p$), and $q'$ is between $b$ and $q$ along $Q$. 

If $pp' \notin E(P)$, then there is an $ef$-theta with $p$ and $p'$ as the two degree 3 vertices. Hence $pp' \in E(P)$ and symmetrically $qq' \in E(Q)$. If $p \neq a$ or $q' \neq b$, then since $\{p,q'\}$ does not separate $e,f$, there is a path in $G-\{p,q'\}$ between the two components of $C-\{p,q'\}$. The ends of this path could be both on $P$ or both on $Q$ or one on each of $P$ and $Q$. In all cases it is routine to check that this path results in an $ef$-theta. Thus we must have $p=a$ and $q'=b$ and similarly $p'=c$ and $q=d$ so $e,f$ are contained in a $K_4$ subgraph of $G$. If $G \neq K_4$, then $G$ has a vertex $x$ not in the $K_4$ subgraph. $G$ is $2$-connected so $G$ has two independent paths from $x$ to distinct vertices of $K_4$ and again we can find an $ef$-theta. Hence the result follows.  \end{proof}

To describe $\theta_{2,t,t}$-free graphs, we use nearly outerplanar graphs. A simple graph $G$ is \textit{nearly outerplanar} if $G$ has a Hamilton cycle $C$ such that every chord crosses at most one other chord and, in addition, if two chords $ab$ and $cd$ do cross, then either both $a,c$ and $b,d$ are adjacent in $C$ or both $a,d$ and $b,c$ are adjacent in $C$. An edge of $C$ is \textit{free} if it does not belong to a $4$-cycle spanned by two crossing chords. A general graph $G$ is {\it nearly outerplanar} if $si(G)$ is nearly outerplanar, and {\it free} edges of $G$ are those that are parallel to a free edge of $si(G)$. For any positive integer $n$, let $\mathcal{O}_n$ be the class of graphs formed by $2$-summing graphs from $\mathcal{L}_n$ to free edges of nearly outerplanar graphs. Note $C(\mathcal{L}_n) \subset \mathcal{O}_n$.  

\begin{theorem} There exist two functions $f_{\ref{thm:2tt}}(t)$ and $g_{\ref{thm:2tt}}(t)$ such that every $2$-connected graph $G$ with $\ell(G) \geq f_{\ref{thm:2tt}}(t)$ either contains $\theta_{2,t,t}$ or is in $\mathcal{O}_{g_{\ref{thm:2tt}}(t)}$, where $t \geq 3$ is an integer. Additionally, all graphs in $\mathcal{O}_t$ are $\theta_{2,t,t}$-free.
\label{thm:2tt}
\end{theorem} 

\begin{proof} As in the proof of Theorem~\ref{thm:1tt}, let $w(t)=f_{\ref{lem:4.2.9}}(2t)$. We will show $f_{\ref{thm:2tt}}(t)=[w(t)+2]^{3t}w(t)$ and $g_{\ref{thm:2tt}}(t)=[w(t)+2]^{t}w(t)$ satisfy the theorem.  
Suppose $\ell(G) \geq f_{\ref{thm:2tt}}(t)$ and further assume $G$ does not contain $\theta_{2,t,t}$. We need to show $G \in \mathcal{O}_{g_{\ref{thm:2tt}}(t)}$.

Let $b=\text{max}\{\ell(G'): G' \text{ is a 3-connected minor of } G\}$; we know $b < w(t)$ since otherwise $G$ contains $\theta_{2,t,t}$. Let $e^*$ be a specified edge of $G$ and consider a chain decomposition of $G$ given by $G_0,G_1, \dots, G_n$ and with vertices $x_0,x_1,\dots,x_n,y_0,y_1,\dots, y_n,z$ as in Figure~\ref{fig:bpath}. If $a(G,e^*) < 3t$, then by Lemma \ref{lem:aGebound}, $G$ can be constructed from its $3$-connected minors and graphs of order $\le 3$ in at most $3t$ iterations of operation $S$. By Lemma~\ref{lem:Sdecompbd}, $\ell(G) \leq (b+2)^{3t} b < [w(t)+2]^{3t}w(t)=f_{\ref{thm:2tt}}(t)$ which is a contradiction. 

Hence assume $a(G,e^*)=n \geq 3t$. Since $G$ is 2-connected, it has a cycle $C^*$ containing $e^*$ and $z$. For each $i \in \{t,t+1,\dots,n-t\}$, let $G_i^+$ be obtained from $G_i$ by adding a new edge $e_i$ between  $x_i,y_i$ and a new edge $f_i$ between $x_{i+1},y_{i+1}$. Then $G_i^+$ is $2$-connected and has no $2$-cut separating $e_i$ from $f_i$ since otherwise we could find a chain decomposition of $G$ with $a(G,e^*) >n$. If $G_i^+$ contains an $e_if_i$-theta $T$, then $C^*\cup (T\backslash \{e_i,f_i\})$ contains a $\theta_{2,t,t}$. Thus by Lemma~\ref{lem:eftheta}, either $e_i,f_i$ have a common end and that end has only two neighbors in $G_i^+$ or $e_i,f_i$ have no common end and $si(G_i^+)=K_4$. We conclude that there exists a nearly outerplanar graph $H$ such that its Hamilton cycle $C$ has length exceeding $t$, and $G$ is obtained from $H$ by 2-summing minors of $G$ to free edges of $C$.

Choose $H$ such that $C$ is as long as possible. Let $G_e$ be a graph $2$-summed to a free edge $e$ of $H$ over an edge $e'$ of $G_e$. In order to conclude $G \in \mathcal{O}_{g_{\ref{thm:2tt}}(t)}$, it suffices to show $\ell(G_e) < g_{\ref{thm:2tt}}(t)$. Suppose $a(G_e,e') = m \ge t$ and let $H_0,H_1,\dots,H_m$ be the corresponding chain decomposition. Let $u,v$ be the two common vertices of $H_0$ and $H_1$. Let $H_0^+$ be obtained from $H_0$ by adding a new edge $f'=uv$. If $H_0^+$ contains an $e'f'$-theta, then $G$ contains $\theta_{2,t,t}$ where one long path goes through $C$ and the other long path goes through $H_1\cup ... \cup H_m$. Thus by Lemma~\ref{lem:eftheta}, either $e',f'$ have a common end and that end has only two neighbors in $H_0^+$ or $e',f'$ have no common end and $si(H_0^+)=K_4$. Each of these two cases contradicts the maximality of $C$. Hence $a(G_e,e') < t$. It follows from Lemmas~\ref{lem:Sdecompbd} and~\ref{lem:aGebound} that $\ell(G_e) \le [b+ 2]^t b < [w(t)+2]^{t}w(t) = g_{\ref{thm:2tt}}(t)$.    

Finally, we prove every $G \in \mathcal{O}_t$ is $\theta_{2,t,t}$-free ($t\ge3$). 
Note every 2-connected minor of every graph in $\mathcal L_t$ remains in $\mathcal L_t$. Similarly, every 2-connected minor of every nearly outerplanar graph remains nearly outerplanar (and free edges remain free). It follows that every 2-connected minor of every graph in $\mathcal O_t$ remains in $\mathcal O_t$. Therefore, to prove every $G\in\mathcal O_t$ is $\theta_{2,t,t}$-free we only need to show $\theta_{2,t,t} \not \in \mathcal O_t$. Suppose otherwise that $\theta_{2,t,t}$ can be formed from a nearly outerplanar graph $H$ by 2-summing $k\ge0$ graphs $H_1,...,H_k\in \mathcal L_t$ to free edges of $H$. Since $\theta_{2,t,t}$ has no 4-cycle, $H$ cannot contain crossing chords and thus $H$ is outerplanar. Let $C$ be the facial Hamilton cycle of $H$ and let $x,y$ be the two cubic vertices of $\theta_{2,t,t}$. Suppose $x \in V(H_i) \setminus V(C)$ for some $i >0$. Then $y \in V(H_i)$ as well because $H_i$ is separated from the rest of the graph by a 2-cut. But now $H_i$ must contain an $xy$-path of length $t$, which contradicts the assumption $H_i\in\mathcal{L}_t$ Therefore, each $H_i$ is a cycle and 2-summing it to $H$ amounts to replacing an edge of $C$ by a path. What this means is that we may consider $H_i$ as part of $C$ in the first place. In other words, we may assume $k=0$. It follows that $\theta_{2,t,t}=H$, which is impossible since $\theta_{2,t,t}$ is not outerplanar. This contradiction completes our proof.
\end{proof}

\section{Planar drawings versus crossing paths}

An important step in proving our main result is to determine if a graph admits a planar drawing with certain vertices and edges on a facial cycle. This problem is essentially solved by Robertson and Seymour in \cite{RoSe90}. However, their result is not strong enough for our application. In the following we first state two results from \cite{RoSe90} and then we prove a refinement of these results. 

Let $C$ be a cycle of $G$. Let $u,v$ be distinct vertices of $G-V(C)$ and let $P_1,P_2,P_3$ be independent $uv$-paths. Then $(P_1,P_2,P_3)$ is a \textit{tripod} of $G$ with respect to $C$ if $G-\{u,v\}$ has three disjoint paths $Q_1,Q_2,Q_3$, where $Q_i$ is from a vertex $s_i$ on $P_i-\{u,v\}$ to a vertex $t_i$ on $C$,  such that either $V(P_i\cap C)=\emptyset$ or  $V(P_i\cap C)=\{s_i\}=\{t_i\}$. The paths $Q_1,Q_2,Q_3$ are {\it legs} and the vertices $t_1,t_2,t_3$ are the \textit{feet} of the tripod. A \textit{cross} of $C$ is a pair of disjoint $C$-paths, one with ends $u,v$ and one with ends $x, y$, such that $u,x,v,y$ appear in that order around $C$. We use the following two lemmas by Robertson and Seymour which we have rephrased using our terminology.

\begin{lemma} [Lemma (2.3) of \cite{RoSe90}] Let $C$ be a cycle of a graph $G$ and let $(P_1,P_2,P_3)$ be a tripod with respect to $C$. If $|C|\ge4$ then either $G$ has a cross with respect to $C$ or $G$ has a $k$-separation $(G_1,G_2)$ with $k\le3$, $V(C)\subseteq V(G_1)$, and $V(P_1 \cup P_2 \cup P_3)\subseteq V(G_2)$. 
\label{lem:2.3RS} 
\end{lemma}    

\begin{lemma} [Lemma (2.4) of \cite{RoSe90}] Let $G$ be $2$-connected with a cycle $C$ of length $\ge3$ such that $G$ has no $2$-separation $(G_1,G_2)$ with $V(C) \subseteq V(G_1)$.  
If $G$ has no cross or tripod with respect to $C$, then $G$ admits a planar drawing with $C$ as a facial cycle. \label{lem:2.4RS} \end{lemma}    

Note in these two lemmas, $C$ has been specified. However, in our applications $C$ will only be partially given. Our problem is to decide if the partial cycle can be completed into a cycle $C$ so that $G$ admits a planar drawing with $C$ as a facial cycle. In the following we make the problem more precise.
 
A \textit{circlet} $\Omega$ of a graph $G$ consists of a cyclically ordered set of distinct vertices $v_1,v_2,\dots,v_n$ of $G$, where $n \geq 4$, and a set of edges of $G$ of the form $v_iv_{i+1}$, where $v_{n+1}=v_1$. Note not necessarily all edges of $G$ of the given form are in $\Omega$. Denote by $V(\Omega)$ and $E(\Omega)$ the set of vertices and edges of $\Omega$, respectively. We call $v_i \in V(\Omega)$ \textit{isolated} if no edge of $\Omega$ is incident with $v_i$. An \textit{$\Omega$-cycle} is a cycle $C$ of $G$ such that $V(\Omega) \subseteq V(C)$, $E(\Omega) \subseteq E(C)$, and the cyclic ordering of $V(\Omega)$ agrees with the ordering in $C$. For each $i$, the $v_iv_{i+1}$-path of $C$ that does not contain $v_{i+2}$ is called a \textit{segment} of $C$. We say $(G, \Omega)$ is \textit{$4$-connected} if $(G, V(\Omega))$ is $4$-connected.   

\begin{theorem} Let $\Omega$ be a circlet of $G$ such that $G$ has an $\Omega$-cycle and $(G, \Omega)$ is $4$-connected. Then either $G$ admits a planar drawing in which some facial cycle is an $\Omega$-cycle, or $G$ has an $\Omega$-cycle $C$ and two crossing paths on $C$ for which each segment of $C$ contains at most two of the four ends of these two crossing paths. 
\label{lem:Scycle}
\end{theorem} 

We need the following two lemmas for proving this theorem. Several different formulations of these lemmas are known, but we were not able to find in the literature the formulation we need. So we prove the lemmas here. Our proofs are similar to that of other versions of the lemmas. 
Let $H$ be a subgraph of $G$ and let $J$ be a subgraph of $H$. An $H$-bridge $B$ is called \textit{$J$-local} if all feet of $B$ are in $J$.  

\begin{lemma} Let $H$ be a subgraph of a simple graph $G$ with $|H|\ge3$. Let $P$ be an $H$-path in $G$ and let $x,y$ be the two ends of $P$. Let $B_1,\dots,B_t$ be all $(H \cup P)$-bridges that are $P$-local. Suppose $G$ has no $k$-separation $(G_1,G_2)$ with $k < 3$ and $V(H) \subseteq V(G_1)$. Then $H_0 = P \cup B_1 \cup \dots \cup B_t$ has an $xy$-path $Q$ such that no $(H \cup Q)$-bridge is $Q$-local.
\label{lem:localbridge}
\end{lemma}

\begin{proof} For any $H$-path $R$ with ends $x,y$, we define $\alpha(R)$ as follows. Let $J_1,\dots,J_n$ ($n \geq 0$) be all $(H \cup R)$-bridges that are $R$-local; let $J_0$ be the union of all other $(H\cup R)$-bridges. Suppose $||J_1|| \geq ||J_2|| \geq \dots \geq ||J_n||$. Then $\alpha(R) = (||J_0||,||J_1||,\dots,||J_n||)$. Among all $xy$-paths in $H_0$, let $Q$ be the path that maximizes $\alpha$ lexicographically. We prove that no $(H \cup Q)$-bridge is $Q$-local.

Suppose otherwise. Let $J_1,\dots,J_n$ ($n \geq 1$) be all $(H \cup Q)$-bridges that are $Q$-local, where $||J_1|| \geq ||J_2|| \geq \dots \geq ||J_n||$, and let $J_0$ be the union of all other $(H \cup Q)$-bridges. Since $G$ has no $k$-separation $(G_1,G_2)$ with $k < 2$ and $V(H) \subseteq V(G_1)$, $J_n$ has at least two feet. Let $a,b$ be the two feet so that the only $ab$-path $Q_{ab}$ of $Q$ contains all feet of $J_n$. Let $L$ be an $ab$-path in $J_n$ that avoids all other feet of $J_n$ and let $Q'$ be obtained from $Q$ by replacing $Q_{ab}$ with $L$. Since $J_n$ is a subgraph of $H_0$, $Q'$ is again an $xy$-path in $H_0$.

Let $Z = V (Q_{ab}-\{a,b\})$. Since $G$ is simple, the choice of $a$ and $b$ implies $Z \neq \emptyset$. Note: $(H \cup Q)$-bridges (other than $J_n$) that have no feet in $Z$ are also $(H \cup Q')$-bridges; $(H \cup Q)$-bridges (other than $J_n$) that have a foot in $Z$ are combined with $Q_{ab}- \{a,b\}$ into a single $(H \cup Q')$-bridge $J^*$ (which may include some subgraphs of $J_n$); and all other $(H \cup Q')$-bridges are subgraphs of $J_n$. Since $|H|\ge3$ and since $G$ has no $k$-separation $(G_1,G_2)$ with $k <3$ and $V(H) \subseteq V(G_1)$, we deduce that at least one $(H \cup Q)$-bridge $J \neq J_n$ has a foot in $Z$. Therefore, $J^*$ contains $J$ and $Q_{a,b}$, implying that at least one of the terms $||J_0||,||J_1||,\dots,||J_{n-1}||$ is increased (since either $J$ is part of $J_0$ or $J$ is some $J_i$ for $i = 1,...,n-1$). What this means is that $\alpha(Q')$ is lexicographically bigger than $\alpha(Q)$, contradicting the maximality of $\alpha(Q)$ and so the lemma is proved. \end{proof} 

Let $H$ be a subdivision of a graph $J$. Then $V(J)$ is a subset of $V(H)$ and $V(J)$-paths of $H$ are exactly the paths obtained by subdividing edges of $J$. We call these paths {\it branches} of $H$. Suppose a subgraph $H$ of $G$ is a subdivision of another graph. Then an $H$-bridge $B$ is called \textit{unstable} if $B$ is $P$-local for a branch $P$ of $H$.

\begin{lemma} Let $G$ contain a subdivision $H$ of $J$ as a subgraph, where $J$ is loopless of order $\ge3$. Suppose $G$ is simple and has no $k$-separation $(G_1,G_2)$ with $k < 3$ and $V(J) \subseteq V(G_1)$. Then $G$ contains a subdivision $H^*$ of $J$ obtained by adjusting branches of $H$ such that all $H^*$-bridges are stable.
\label{lem:stablebridge}
\end{lemma} 

\begin{proof} We first replace each branch of $H$ by a single edge of $G$ whenever it is possible. 
Then we repeatedly apply Lemma~\ref{lem:localbridge} to every branch of $H$. Note after each application of Lemma~\ref{lem:localbridge}, no new unstable bridge is created. Therefore, after the final step all bridges are stable. \end{proof}

\begin{proof} [Proof of Theorem~\ref{lem:Scycle}] Assume $G$ does not have a planar drawing in which some facial cycle is an $\Omega$-cycle. We will show $G$ has an $\Omega$-cycle and two crossing paths on the cycle that satisfy the theorem. Without loss of generality we assume $G$ is simple.

Let $C$ be an $\Omega$-cycle of $G$. By Lemma~\ref{lem:stablebridge} we assume no segment of $C$ contains all feet of any $C$-bridge. Since $G$ does not have a desired planar drawing, by Lemma~\ref{lem:2.4RS}, $G$ has either two crossing paths or a tripod on $C$. By Lemma~\ref{lem:2.3RS}, if $G$ has a tripod, then it also has two crossing paths (since $(G,\Omega)$ is $4$-connected) so let $Q_1, Q_2$ be crossing paths on $C$. Let $x_1,x_3$ be the ends of $Q_1$ and $x_2,x_4$ be the ends of $Q_2$. Suppose for the sake of contradiction some segment $P$ of $C$ contains more than two of $x_1,x_2,x_3,x_4$. Let $v_1,v_2$ be the ends of $P$.

Suppose first that $x_1,x_2,x_3,x_4 \in P$. Let $B$ be the $C$-bridge that contains $Q_1$ and let $x$ be a foot of $B$ not on $P$. Let $Q$ be a path in $B$ from $x$ to the interior of $Q_1$ (or $Q_1 \cup Q_2$ if $B$ also contains $Q_2$). Then $Q_1 \cup Q_2 \cup Q$ contains two crossing paths on $C$ so that $P$ contains only three of the four ends. Hence without loss of generality, we can assume $P$ contains $x_1,x_2,x_3$ but not $x_4$. Again let $B$ be the $C$-bridge that contains $Q_1$. Then $B$ contains a path $Q$ from the interior of $Q_1$ to a foot of $B$ not on $P$. If $Q$ is disjoint from $Q_2$, then $Q_1 \cup Q_2 \cup Q$ contains the desired crossing paths. If $Q$ meets $Q_2$, say at a vertex $y$, then let $G'=G+v_1v_2$ and let $C'$ be the cycle of $G'$ obtained by replacing $P$ with $v_1v_2$. Now $Q_1 \cup Q_2 \cup Q \cup P$ contains a tripod $T$ with respect to $G'$ and $C'$; the feet of $T$ are $v_1,v_2,x_4$. Without loss of generality, assume $T$ is a tripod with feet $v_1,v_2,x_4$ such that the legs $P_1$ from $v_1$ to $x_1$, $P_2$ from $v_2$ to $x_3$, and $P_3$ from $x_4$ to $y$ are minimal. Since $\{x_1,x_3,y\}$ is not a $3$-cut of $G'$, there is a path $R$ of $G'-\{x_1,x_3,y\}$ from $T$ to $C' \cup P_1 \cup P_2 \cup P_3$. By the minimality of $P_1,P_2,P_3$, we know $R$ ends at $C'-\{v_1,v_2,x_4\}$. Now an $\Omega$-cycle $C''$ can be obtained from $C'$ by replacing $v_1v_2$ with a path in $P_1\cup P_2\cup T$, and desired crossing paths on $C''$ can be obtained from $P_3\cup T \cup R$.  \end{proof}

When $\Omega$ has no isolated vertices, we can further strengthen Theorem~\ref{lem:Scycle}. 

\begin{theorem} Let $\Omega$ be a circlet of $G$ such that $\Omega$ has no isolated vertices, $|E(\Omega)| \geq 3$, $G$ has an $\Omega$-cycle, and $(G,\Omega)$ is $4$-connected. Then either $G$ admits a planar drawing in which some facial cycle is an $\Omega$-cycle, or $G$ has an $\Omega$-cycle $C$ and two crossing paths on $C$ for which among the four paths of $C$ divided by the four ends of the two crossing paths, at least three of them contain an edge of $\Omega$. 
\label{lem:Scycleedge} 
\end{theorem}

\begin{proof} By Theorem~\ref{lem:Scycle} we assume $G$ has an $\Omega$-cycle $C$ and two crossing paths $P_1,P_2$ on $C$ for which each segment of $C$ contains at most two of the four ends $x_1,x_2,x_3,x_4$ of $P_1,P_2$. We need to show $G$ has two crossing paths on an $\Omega$-cycle that satisfy the theorem.  

Assume $x_1,x_2,x_3,x_4$ appear in that forward order around $C$. Let $Q_i =C[x_i,x_{i+1}]$ for $i=1,2,3$ and  $Q_4=C[x_4,x_1]$. Suppose to the contrary that at most two of the $Q_i$ contain edges of $\Omega$. Then the choice of $P_1,P_2$ and the assumption that $\Omega$ has no isolated vertices imply that exactly two of the $Q_i$ contain edges of $\Omega$ and these two $Q_i$ cannot be adjacent. Without loss of generality, suppose $Q_1$ and $Q_3$ contain edges of $\Omega$. Since $|E(\Omega)| \geq 3$, we further assume $Q_1$ contains at least two edges of $\Omega$.

Re-choose (if necessary) $C, P_1,P_2$ so that $Q_1$ is as short as possible. Since $G$ is $3$-connected, $G- \{x_1,x_2\}$ has a path $R$ from $Q_1 - \{x_1,x_2\}$ to $(C \cup P_1 \cup P_2) - V(Q_1)$. Let $v$ be the endpoint of $R$ on $Q_1$; then $v$ is between two edges of $\Omega$ since otherwise the minimality of $Q_1$ is violated. If the other end of $R$ is on $P_1\cup P_2$, then $R\cup P_1\cup P_2$ contains the desired two crossing paths. If the other end of $R$ is on $C$, then $R$ and one of $P_1,P_2$ form the desired crossing paths. 
\end{proof}

We close this section by proving the following technical lemma which we will use in the next section.

\begin{lemma}
\label{lem:tri-ext}
Suppose a $3$-connected graph $G$ has a triangle $T$ and edge $e$ such that at most one end of $e$ is in $T$. 
Then either $G$ contains one of the two graphs in Figure \ref{fig:nonplanar} as a minor or $G=S_3(G_0;G_1,...,G_k)$ where $G_0$ is planar with $T$ as a facial cycle and each $G_i$ $(i>0)$ has order $\ge 5$ and is $3$-summed to a facial triangle of $G_0$ different from $T$. 
\end{lemma}

\begin{figure}[ht]
\centerline{\includegraphics[scale=0.7]{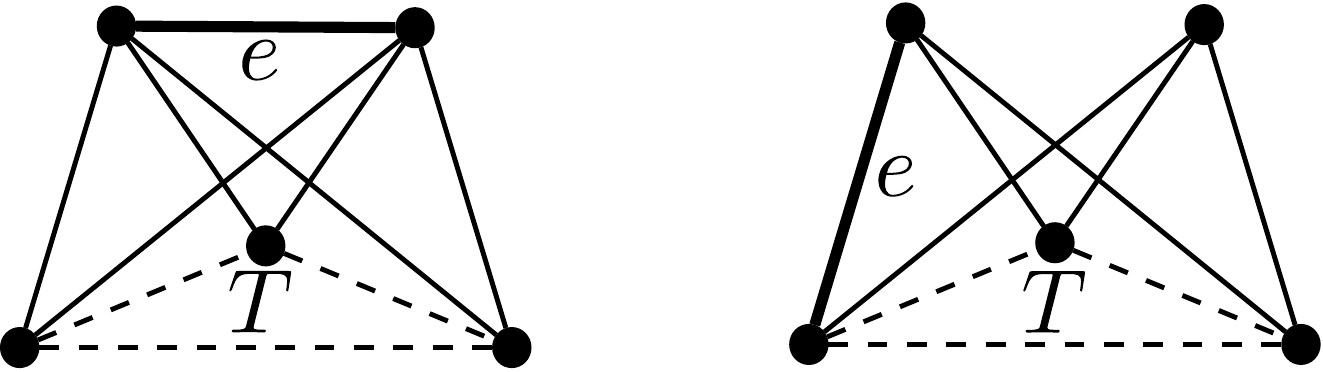}}
\caption{Two nonplanar minors $A_1$ and $A_2$ \label{fig:nonplanar}}
\end{figure} 

\begin{proof}
We first make an observation: {\it if $Z\subseteq V(G)$ contains at most one end of $e$ and $|Z|\ge3$, then $G$ has three independent paths from a vertex outside $Z$ to three distinct vertices of $Z$ such that $e$ is on one of these paths.} To see this, first find two disjoint paths from the two ends of $e$ to $Z$. These paths and $e$ form a $Z$-path $P$ containing $e$. Let $z_1,z_2$ be the two ends of $P$. Then $G-\{z_1,z_2\}$ has a path $Q$ from $Z-\{z_1,z_2\}$ to $P-\{z_1,z_2\}$. It follows that $P\cup Q$ is the union of the three required paths. 

If $V(T)$ is a 3-cut of $G$ then we deduce from the above observation by taking $Z=V(T)$ that $G$ contains an $A_2$ minor. Assume $V(T)$ is not a 3-cut. By Lemma \ref{lem:3sum}, $G$ has 3-connected minors $G_0,G_1,...,G_k$ such that $G=S_3(G_0;G_1,...,G_k)$, where $T\subseteq G_0$, $(G_0,V(T))$ is 4-connected, and $|G_i|\ge5$ for all $i>0$. Suppose $G_1,...,G_k$ are chosen to be maximal. Then we may assume that $G_0$ is nonplanar because otherwise $G_0,G_1,...,G_k$ satisfy the requirements of the lemma. 

We claim that $G_0$ has three independent $uv$-paths $P_1,P_2,P_3$, for some $u,v$ outside $T$, such that $T$ meets all of these three paths. To see this, first note by Lemma \ref{lem:2.4RS}, $G_0$ has a tripod $(P_1,P_2,P_3)$ on $T$. Let $Q_i,s_i,t_i$ be determined as in the definition of tripod. We choose the tripod with $Q_1\cup Q_2\cup Q_3$ as small as possible. If $s_i\ne t_i$ for some $i$, say for $i=1$, then, as $(G_0,V(T))$ is 4-connected, $G_0-\{s_1,s_2,s_3\}$ has a path $P$ from $P_1\cup P_2\cup P_3$ to $T$. It is routine to see that the union of $P$ and all $P_i$ and $Q_i$ contains a tripod with shorter legs. This contradiction shows $s_i=t_i$ for all $i$ and thus our claim follows.

Now we consider two cases. First, suppose both ends of $e$ are in $Z=V(P_1\cup P_2\cup P_3)$. Then it is straightforward to verify that either $A_1$ or $A_2$ is a minor of $G$. Now in the second case, we assume $Z$ contains at most one end of $e$. By our earlier observation, $G$ has three independent paths $R_1,R_2,R_3$ from a vertex outside $Z$ to $Z$ such that $e$ is on one of these paths. If $V(R_1\cup R_2\cup R_3) \cap Z=V(T)$ then $G$ contains $A_2$ as a minor. If $V(R_1\cup R_2\cup R_3) \cap Z\ne V(T)$ then $R_1\cup R_2\cup R_3$ contains a $Z$-path $R$ such that $R$ contains $e$ and at least one end of $R$ is not in $T$. This situation reduces to our first case and thus $G$ contains the required minor.
\end{proof}


\section{$3$-connected $\theta_{t,t,t}$-free graphs}

In this section we focus on 3-connected graphs. Let $(G_0,w_0)$ be a weighted plane graph and let $(G_1,w_1),\dots,(G_k,w_k)$ be disjoint weighted graphs with $|G_i|\ge 5$ for all $i>0$. Denote by $S_3^p((G_0,w_0)$; $(G_1,w_1)$, ..., $(G_k,w_k))$ a weighted graph $(G,w)$ obtained by $3$-summing $(G_1,w_1)$, ..., $(G_k,w_k)$ to inner facial triangles of $(G_0,w_0)$. Let $r,s\ge2$ be integers. Let $\mathcal L_{r,s}^3$ be the class of $3$-connected members of $\mathcal L_{r,s}$. Let $\mathcal P_r^3$ be the class of 3-connected members $(G,w)\in\mathcal P_r$ such that if $C$ is the outer cycle of $G$ then either $|C|\ge 3r$ or $C$ contains at least three edges of weight at least $r$. Let $\Phi^3(\mathcal L_{r,s}^3, \mathcal P_r^3)$ be the class of $3$-connected weighted graphs of the form $S_3^p((G_0,w_0)$; $(G_1,w_1)$, $\dots$, $(G_k,w_k))$ ($k \geq 0$) over all $(G_0,w_0) \in \mathcal P_r^3$ and $(G_1,w_1)$, $\dots$, $(G_k,w_k) \in \mathcal L_{r,s}^3$ with $|G_i|\ge5$ for all $i>0$. In the rest of the paper we will call an edge {\it heavy} if its weight is at least $t$. The following is the main result of this section. 

\begin{theorem} There exists a function $f_{\ref{thm:3need}}(t)$ such that if $(G,w)$ is $3$-connected and $\theta_{t,t,t}$-free, then one of the following holds. \\ 
\indent $(a)$ $(G,w) \in \Phi^3(\mathcal L_{t,f_{\ref{thm:3need}}(t)}^3,\mathcal P_t^3)$, \\ 
\indent $(b)$ $G \in \mathcal{L}_{f_{\ref{thm:3need}}(t)}^3$ and either $G$ has at most two heavy edges or $G$ has exactly three edges and these three form a triangle.
\label{thm:3need} 
\end{theorem} 

The proof of this theorem is divided into three steps. The first two are given in two lemmas, which deal with  unweighted graphs. For any integer $k \geq 2$, let $W_k^+$ be the graph obtained from $W_{2k}$ with rim cycle $x_1x_2 ... x_{2k}x_1$ by first subdividing the edges $x_1x_2$ and $x_{k+1}x_{k+2}$ and then joining these two new vertices by an edge. 
Let $W_k'$ be obtained from $W_k$ by adding a parallel edge to each of its spokes. We define $W_k'$ for technical purpose because now $W_k'$ is the edge-disjoint union of $k$ triangles and thus we can talk about 3-summing graphs to all these triangles.

\begin{lemma} There exists a function $f_{\ref{lem:ttt}}(t,k)$ such that every $3$-connected graph with a path of length $f_{\ref{lem:ttt}}(t,k)$ either contains $W_t^+$ or $L_t^+$ as a topological minor or can be expressed as $S_3(W_k'; G_1,\dots,G_k)$, where $t\ge2$ and $k\ge4$ are integers and $|G_i|\ge5$ for all $i$. 
\label{lem:ttt}
\end{lemma}  

\begin{proof} Let $f_R(t)$ be the minimum integer such that every connected simple graph on at least $f_R(t)$ vertices has an induced $K_{t+2}$, $K_{1,3}$, or $P^{2t+2}$ (such a function arises as an extension of Ramsey theory and its existence was proven in \cite{ding1}).  We prove $f_{\ref{lem:ttt}}(t,k)=f_{\ref{lem:4.2.12}}(3k(t+1)^2f_R(t))$ satisfies the lemma. Let $G$ be a $3$-connected graph with $\ell(G)\ge f_{\ref{lem:ttt}}(t,k)$. We assume $G$ is simple and $G$ does not contain $L_t^+$ as a topological minor. Then by Lemma~\ref{lem:4.2.12}, $G$ has a subgraph $H$ isomorphic to a subdivision of $W_n$ where $n \geq 3k(t+1)^2f_R(t)$. Take $n$ to be maximal.      

Let $x_0, x_1, \dots, x_n$ be the non-subdividing vertices of $H$ with $x_0$ corresponding to the center. For $i=1,2,\dots,n$, let $P_i$ be the $x_0x_i$-path and $Q_i$ be the $x_ix_{i+1}$-path (where $x_{n+1}=x_1$) of $H$. By Lemma~\ref{lem:stablebridge}, we may assume the feet of each $H$-bridge are not contained in a single $P_i$ or $Q_i$. Let $E_0= E((P_1\cup ... \cup P_n)-x_0)$; let $G'= (G-x_0)/E_0$ and $H' = (H-x_0)/E_0$. To simplify our notation, we consider each $Q_i$ as a path of $H'$ as well. Note $H'$ is the cycle formed by the union of all paths $Q_i$, and because no trivial $H$-bridge has a foot at $x_0$, there is a one-to-one correspondence between $H$-bridges of $G$ and $H'$-bridges of $G'$. Moreover, since $G$ is $3$-connected, and by the choices of each $P_i$ and $Q_i$, each $H'$-bridge of $G'$ has at least two feet on $H'$.  

For any path $J$ of $H'$, define the {\it $Q$-length} of $J$ to be the least number of paths $Q_i$ whose union contains $J$. Suppose $G'$ has an $H'$-bridge $B$ that contains two feet $u,v$ for which both $uv$-paths of $H'$ are of $Q$-length $\ge t+1$. Then $H\cup B$ contains $W_t^+$ as a topological minor since $n\ge 2t+2$. Hence assume any two feet of any $H'$-bridge are contained in a path of $H'$ of $Q$-length $\le t$. Since $n>3t$, it follows that all feet of any $H'$-bridge are contained in a path of $H'$ of $Q$-length $\le t$. For each $H'$-bridge $B$, let $Q(B)$ denote the unique minimal path of $H'$ of $Q$-length $\le t$ that contains all feet of $B$. Generally, as $n$ is much bigger than $t$, we can think of each path $Q(B)$ as a very small segment of $H'$; this leads to a rough description of $G'$ as a long cycle with bridges attached to small segments of the cycle.  

To understand the structure of $G'$, we do not need to know all $H'$-bridges. Instead, knowing the ``maximal'' ones will be enough. Let $\mathcal{B}$ be a minimal set of $H'$-bridges such that for every $H'$-bridge $B_1$, there exists $B_2 \in \mathcal{B}$ with $Q(B_1) \subseteq Q(B_2)$. We will focus on bridges in $\mathcal{B}$. Let $\Gamma$ be the simple graph with vertex set $\mathcal{B}$ such that $B_1$ and $B_2$ are adjacent if $E(Q(B_1) \cap Q(B_2)) \neq \emptyset$. For any subgraph $\Gamma'$ of $\Gamma$, we will say the \textit{bridges of $\Gamma'$} to mean the bridges corresponding to the vertices of $\Gamma'$. 

Suppose a component $\Gamma'$ of $\Gamma$ has at least $f_R(t)$ vertices. Because of the way in which $\Gamma$ was constructed, $\Gamma$ does not contain any induced claws; therefore $\Gamma'$ contains an induced $K_{t+2}$ or $P^{2t+2}$. If $\Gamma'$ contains an induced $K_{t+2}$, then $H'$ together with bridges of this clique contains a subdivision of the M\"obius ladder as shown in Figure~\ref{fig:K2t}, where each bridge $B_i$ is represented by a chord joining the two ends of $Q(B_i)$. As a result, $G'$ and hence $G$ contains $L_t^+$ as a topological minor. Similarly, if $\Gamma'$ contains an induced $P^{2t+2}$, then $H'$ together with bridges of this path contains $L_t^+$ as a topological minor. 
Thus we conclude each component of $\Gamma$ has fewer than $f_R(t)$ vertices. 
      
\begin{figure}[ht]
\centerline{\includegraphics[scale=0.4]{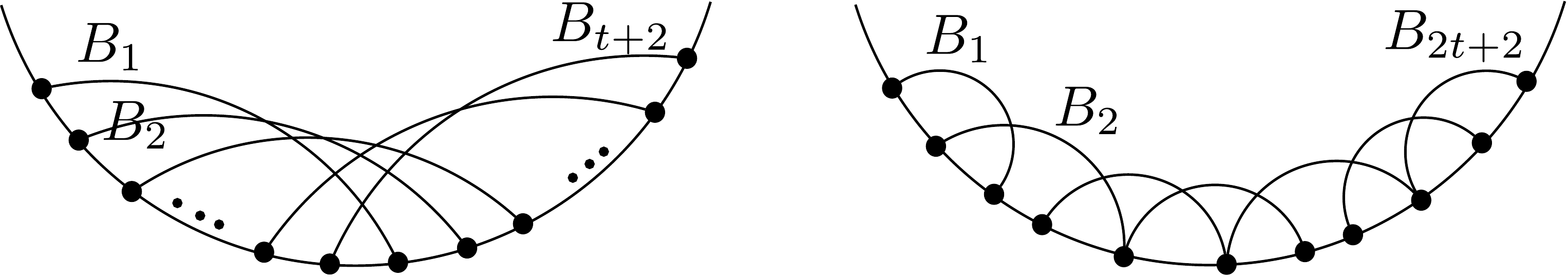}}
\caption{$\Gamma'$ contains an induced $K_{t+2}$ or $P^{2t+2}$\label{fig:K2t}}
\end{figure}

For each component $\Gamma'$ of $\Gamma$, let $Q(\Gamma')$ be the union of $Q(B)$ over all bridges $B$ of $\Gamma'$. Then $Q(\Gamma')$ is a path of $H'$ and its $Q$-length is less than $tf_R(t)$. Since  $n$ is much bigger than $tf_R(t)$, these paths again can be viewed as very short segments of $H'$. Let $\Gamma_1,\Gamma_2$ be distinct components of $\Gamma$. Observe $Q(\Gamma_1)$ and $Q(\Gamma_2)$ are edge-disjoint. We say $\Gamma_1,\Gamma_2$ are {\it linked} if $Q(\Gamma_1)$ and $Q(\Gamma_2)$ have a common end $v$ such that $v$ is obtained by contracting $E(P_i-x_0)$ for some $i$, and for each $j\in\{1,2\}$, there is a bridge $B_j$ for which, when viewed as an $H$-bridge of $G$, $B_j$ has a foot in $P_i-\{x_0,x_i\}$, and when viewed as an $H'$-bridge of $G'$, $B_j$ has a foot in $Q(\Gamma_j)-v$.

A {\it linkage} $\Lambda$ is a maximal sequence $\Gamma_1, ..., \Gamma_m$ of components of $\Gamma$ such that $Q(\Gamma_i)$ and $Q(\Gamma_{i+1})$ are linked for $i=1,...,m-1$. Suppose there is a linkage $\Lambda$ with $m \ge 4$. Let us consider each $\Gamma_i$ with $2\le i\le m-1$. Let the two ends of $Q(\Gamma_i)$ be obtained by contracting $P_r-x_0$ and $P_s-x_0$; let $B_r,B_s$ be bridges linking $P_r-\{x_0,x_r\}$ and $P_s-\{x_0,x_s\}$, respectively, to the rest of $Q(\Gamma_i)$, as shown in Figure~\ref{fig:laddertype1}. Note $B_r,B_s$ may not belong to $\cal B$ (and $B_r$ in the Figure is such an example). Choose two bridges of $\Gamma_i$ so that the two ends of $Q(\Gamma_i)$ are feet of these two bridges, respectively. In our example $B_r'$ and $B_s$ are these two bridges. Since $\Gamma_i$ is connected, it contains an induced path between these two bridges. Then bridges of this path together with $B_r,B_s$, and $Q(\Gamma_i)$ contain two disjoint paths $R_i',R_i''$ of $G$ between $P_r-x_0$ and $P_s-x_0$. Now it is easy to see that the union of $R_i',R_i''$ ($i=2,...,m-1$) and $H-x_0$ contains $L_{m-3}^+$ as a topological minor. 

\begin{figure}[ht]
\centerline{\includegraphics[scale=0.6]{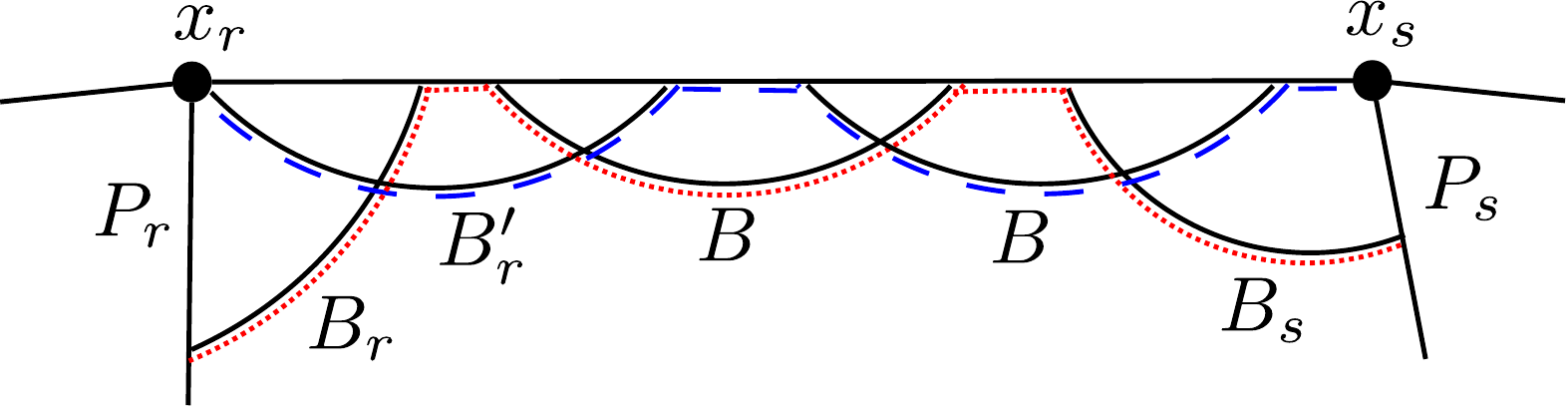}}
\caption{$Q(\Gamma_i)$ and some relevant bridges  
\label{fig:laddertype1}}
\end{figure}

What we have shown is that each linkage can have at most $t+2$ terms. Let $Q(\Lambda)$ denote the union of $Q(\Gamma_i)$ over all terms $\Gamma_i$ of $\Lambda$. Then $Q(\Lambda)$ is a path of $H'$ with $Q$-length $< t(t+2)f_R(t)$. Let $I_\Lambda$ consist of all $i$ such that either $x_i$ is an interior vertex of $Q(\Lambda)$ or $x_i$ is an end of $Q(\Lambda)$ for which $G$ has an $H$-bridge with feet in both $P_i-\{x_0,x_i\}$ and $Q(\Lambda)-x_i$. Let $Q^+(\Lambda)$ be the union of $Q(\Lambda)$ (as a path of $H$) and $P_i$ for all $i\in I_\Lambda$. The four shaded subgraphs in Figure \ref{fig:wheel} are examples of $Q^+(\Lambda)$.  
For any two distinct linkages $\Lambda_1,\Lambda_2$, since $Q(\Lambda_1)$ and $Q(\Lambda_2)$ are edge-disjoint, it follows that $Q^+(\Lambda_1)$ and $Q^+(\Lambda_2)$ are also edge-disjoint. Moreover, the only possible common vertices of $Q^+(\Lambda_1)$ and $Q^+(\Lambda_2)$ are $x_0$ and the common end of $Q(\Lambda_1)$ and $Q(\Lambda_2)$.

We claim that for every $H$-bridge $B$ there exists a linkage $\Lambda$ such that all feet of $B$ are contained in $Q^+(\Lambda)$. When $B$ is viewed as an $H'$-bridge, $Q(B)$ is contained in $Q(B')$ for some $B'\in \cal B$ and thus $Q(B)$ is contained in $Q(\Lambda)$ for a linkage $\Lambda$. Then the definition of $I_\Lambda$ implies that, when $B$ is viewed as an $H$-bridge, all feet of $B$ are in $Q^+(\Lambda)$, which proves our claim.

\begin{figure}[ht]
\centerline{\includegraphics[scale=0.4]{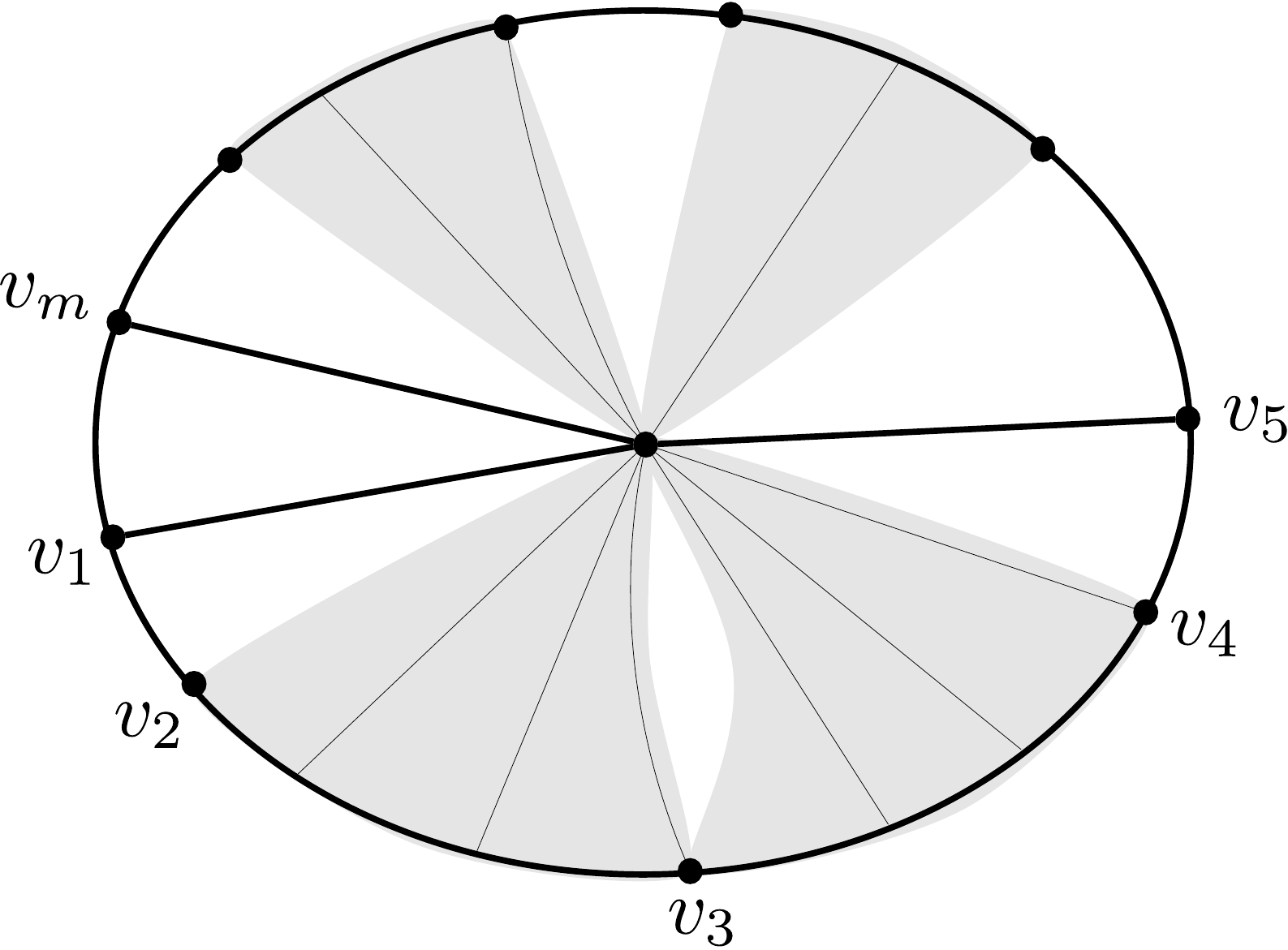}}
\caption{$H$ is divided according to $H$-bridges  \label{fig:wheel}}
\end{figure}

For each vertex $v$ of $H'$, if $v$ is a foot of at least one $H'$-bridge then $v$ is contained in $Q(B)$ for at least one $B\in\cal B$. Let $Z$ be the set of vertices $z$ of $H'$ such that $z$ is not contained in $Q(B)$ for any $B\in \cal B$. Then for each $z\in Z$ there exists $i$ such that $z=x_i$, $P_i$ contains only one edge $x_0x_i$, and $x_i$ has degree 3 in $G$. In Figure \ref{fig:wheel}, $Z$ consists of $v_1,v_5,v_m$. 
It follows that every vertex of $H'$ belongs to either $Z$ or $Q(\Lambda)$ for some linkage $\Lambda$. 
Let $Y$ be the set of vertices $y$ on the rim of $H$ such that there is a linkage $\Lambda$ for which, when $Q(\Lambda)$ is considered as a path of $H$, $y$ is an end of this path. In our example, $Y$ contains seven vertices including $v_2,v_3,v_4$. Let $v_1,v_2,...,v_m$ be all vertices of $Y\cup Z$, which are listed in the order they appear on the rim cycle of $H$. Now we verify  $G=S_3(W_m';H_1,...,H_m)$, where $W_m'$ contains $x_0$ as its center and cycle $v_1v_2...v_mv_1$ as its rim. In fact, if $v_i,v_{i+1}$ (where $v_{m+1}=v_1$) are the two ends of some $Q(\Lambda)$, then by our claim from the last paragraph, the graph consists of $Q^+(\Lambda)$ and all $H$-bridges with feet in $Q^+(\Lambda)$ are attached to triangle $x_0v_iv_{i+1}$ of $W_m'$. Since every $H$-bridge is attached to some $Q^+(\Lambda)$, for every other triangle of $W_m'$, no extra graph is attached to it. Thus $G=S_3(W_m';H_1,...,H_m)$, as required. Now it is clear that by taking a smaller wheel on vertices $x_0, v_1, v_4, ..., v_{\lfloor m/3\rfloor -2}$ we have $G=S_3(W_{\lfloor m/3\rfloor}';G_1,...,G_{\lfloor m/3\rfloor})$ and such that $|G_i|\ge5$ for all $i$.

It remains to show that $|Y\cup Z|\ge 3k$. 
We assume $|Z|<3k$ because otherwise we are done. We prove that there are at least $3k$ linkages, which would imply $|Y|\ge 3k$. Suppose otherwise. Since each $Q(\Lambda)$ has $Q$-length $< t(t+2)f_R(t)$, at most $t(t+2)f_R(t)$ vertices $x_i$ are contained in each $Q(\Lambda)$. It follows that the total number of vertices $x_i$ would be $< |Z| + 3kt(t+2)f_R(t) < 3k(t^2+2t+1)f_R(t) =n$. This contradiction completes our proof of the lemma.
\end{proof}

To simplify our notation, for any class $\mathcal G$ of weighted graphs, we will write $G\in\cal G$ if $(G,\varepsilon)\in\mathcal G$, where $\varepsilon(e)=1$ for all edges $e$ of $G$. Using this terminology, $G\in\mathcal P_r^3$ is equivalent to: $G$ is a 3-connected plane graph such that if $C$ is the outer cycle then $|C|\ge 3r$ and $G$ has no $C$-path of length at least $2r$. Note wheels are examples of such graphs. Let $\mathcal{L}_s^3$ denote the class of $3$-connected graphs in $\mathcal L_s$. Then $G\in \mathcal L_s^3$ if and only if $G\in\mathcal L_{r,s}^3$. Finally, both $S_3^p$ and $\Phi^3$ can be naturally restricted to unweighted graphs. That is, $S_3^p(G_0;G_1,...,G_k)$ is a graph obtained by 3-summing $G_1,...,G_k$, each of order $\ge5$, to inner facial triangles of a plane graph $G_0$, and $\Phi^3(\mathcal L_s^3,\mathcal P_r^3)$ is the class of $3$-connected graphs of the form $S_3^p(G_0;G_1,\dots,G_k)$ ($k \geq 0$) over all $G_0 \in \mathcal{P}_r^3$ and $G_1,\dots,G_k \in \mathcal{L}_s^3$ of order $\ge5$.

\begin{lemma} There exists a function $f_{\ref{thm:ttt}}(t)$ such that all $3$-connected $\theta_{t,t,t}$-free graphs belong to $\mathcal L_{f_{\ref{thm:ttt}}(t)}^3 \cup \Phi^3(\mathcal L_{f_{\ref{thm:ttt}}(t)}^3,\mathcal P^3_t)$. 
\label{thm:ttt}
\end{lemma}

\begin{proof} We show $f_{\ref{thm:ttt}}(t)=f_{\ref{lem:ttt}}(2t,3t)$ satisfies the theorem. For simplicity, let $s(t)=f_{\ref{thm:ttt}}(t)$. Suppose $G$ is a $3$-connected $\theta_{t,t,t}$-free graph that does not belong to $\mathcal{L}_{s(t)}^3$. We will show that $G \in \Phi(\mathcal L_{s(t)}^3,\mathcal P_t^3)$. Since both $W_{2t}^+$ and $L_{2t}^+$ contain $\theta_{t,t,t}$, by Lemma~\ref{lem:ttt}, $G$ can be expressed as $S_3(W_{3t}';G_1,\dots,G_{3t})$, where $|G_i|\ge5$ for all $i$. It follows that $G$ can be expressed as $G=S_3^p(G_0; H_1,...,H_h)$, where $G_0, H_1,..., H_h$ are 3-connected minors of $G$, $|H_i|\ge5$ for all $i$, $G_0$ is planar, and $G_0$ has a subgraph $H_0$ such that $H_0$ is a subdivision of $W_k$ with $k \geq 3t$ and the rim cycle of $H_0$ is a facial cycle of $G_0$. Choose $G_0$ so that $|G_0|$ is as big as possible. Let $x_0,x_1,\dots,x_k$ be the non-subdividing vertices of $H_0$ with $x_0$ corresponding to the center. By Lemma  \ref{lem:removeC}, $G_0\in \mathcal P_t^3$. So we only need to show $H_i \in \mathcal{L}_{s(t)}^3$ for all $i$.

To simplify notation, assume $i=1$. We suppose $H_1$ has a path of length $s(t)$ and derive a contradiction. Let $y_0y_1y_2$ be the common triangle of $G_0$ and $H_1$. Note $y_0y_1y_2$ is a face of $G_0$ so it is contained in some face of $H_0$. Let $C$ be the cycle bounding the region containing $y_0y_1y_2$ where $C$ corresponds to triangle $x_0x_1x_2$ of $H_0$. Since $G_0$ is $3$-connected, there are three disjoint paths in $G_0$ (in fact, inside $C$) from $x_0x_1x_2$ to $y_0y_1y_2$. By renaming the indices of $y_0y_1y_2$, if necessary, we assume that the paths are from $x_i$ to $y_i$ ($i=0,1,2$). Note the $x_0y_0$-path is disjoint from the rim of $H_0$. 

Suppose at least one of $y_1,y_2$, say $y_2$, is not on the rim of $H_0$. Since $H_1$ is $3$-connected, $H_1-y_2$ is $2$-connected. Since $H_1$ has a path of length $s(t)$ (and $s(t) =f_{\ref{lem:ttt}}(2t,3t) > 8t^2$), $H_1-y_2$ has a path of length $4t^2$ and hence by Lemma~\ref{lem:lemmaa}, a $y_0y_1$-path $P$ of length at least $t$. Now we have a contradiction since $G_0 \cup P$ contains $\theta_{t,t,t}$ at $x_0$ and $x_1$: one path uses $P$ as well as the $x_0y_0$-path and $x_1y_1$-path, and the other two paths are in $H_0$. It is important to note edges of triangle $y_0y_1y_2$ are not used in this $\theta_{t,t,t}$ since these three edges are deleted when $H_1$ is 3-summed to $G_0$. 

From the last paragraph we conclude that both $y_1,y_2$ are on the rim of $H_0$. Since $G_0$ is 3-connected, $y_1$ and $y_2$ must be adjacent in $H_0$. We assume that $y_1y_2$ is an edge of $H_0$ and, moreover, $G_0$ has no other edges parallel to $y_1y_2$ since all such edges can be placed in $H_1$. In the following we will look, in $H_1$, for a path from $y_1$ to $y_2$ together with a path $P$ of length at least $t$ from this path to $y_0$; call $P$ a \textit{long spoke}. With these two paths, there is a $\theta_{t,t,t}$ in $G$ at $x_0$ and $v$ as shown in Figure \ref{fig:G_0}.

\begin{figure}[ht]
\centerline{\includegraphics[scale=0.6]{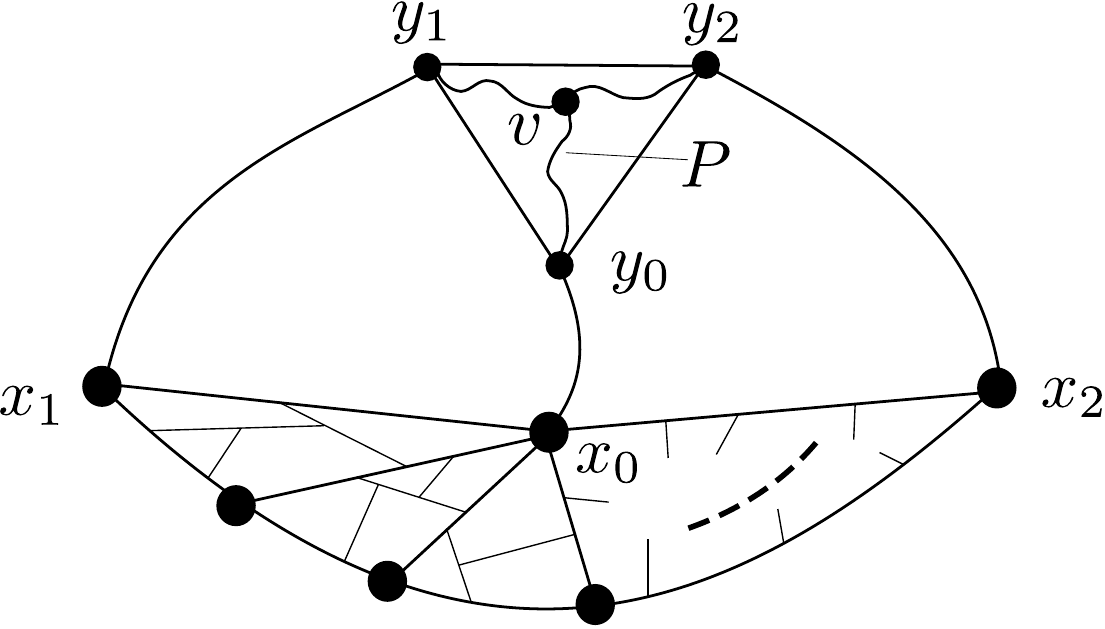}}
\caption{a long spoke in $G$ \label{fig:G_0}}
\end{figure} 

Since $H_1$ is $\theta_{t,t,t}$-free with $\ell(H_1)\ge s(t)$, by Lemma~\ref{lem:ttt}, $H_1=S_3(J_0;J_1,\dots,J_{3t})$ where $J_0=W_{3t}'$ and $|J_i|\ge5$ for all $i>0$. Let $z_0$ be the center of $J_0$ and $z_1z_2....z_{3t}z_1$ be its rim cycle. Without loss of generality, assume $y_0y_1y_2$ is contained in $J_1$ and $z_0z_1z_2$ is the common triangle of $J_0$ and $J_1$. 

Since $J_1$ is $3$-connected, there are three disjoint paths $P_i$ ($i=0,1,2$) from $z_i$ to the triangle $y_0y_1y_2$. Suppose the other end of $P_0$ is not $y_0$. Then $H_1$ contains three independent paths $Q_0, Q_1, Q_2$ from $z_0$ to $y_0,y_1,y_2$, respectively, as shown in the left in Figure~\ref{fig:H_1}. Since $Q_0$ has length at least $t$, it is a long spoke and $G$ contains a $\theta_{t,t,t}$. Hence assume $P_i$ is from $z_i$ to $y_i$ ($i=0,1,2$) as on the right in Figure~\ref{fig:H_1}. 

\begin{figure}[ht]
\centerline{\includegraphics[scale=0.6]{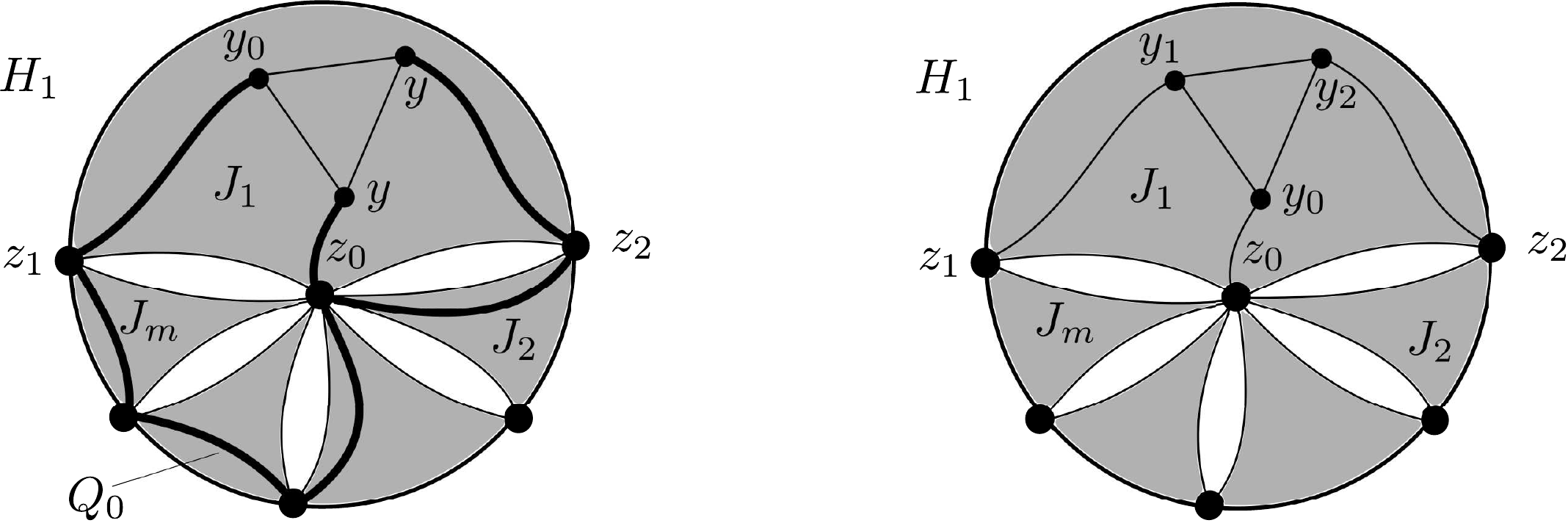}}
\caption{decomposition of $H_1$ into pieces \label{fig:H_1}}
\end{figure} 

Let $H_1'$ be the 3-sum of $J_0$ and $J_1$. In other words, $H_1'$ is obtained from $H_1$ by reducing each $J_i$ ($i>1$) to a triangle. Then $H_1'$ is 3-connected. Let $\Omega$ be the circlet of $H_1'$ with vertices $z_1,y_1,y_0,y_2, z_2, z_3, ...,z_{3t}$, which are cyclically ordered as they are listed, and with $3t+1$ edges from the two paths $y_1y_0y_2$ and $z_2z_3...z_{3t}z_1$. Note $\Omega$ is well-defined even if $y_1=z_1$ or $y_2=z_2$. From Lemma~\ref{lem:3sum}, we know that $H_1'$ has 3-connected minors $M_0, M_1,\dots, M_a$ such that $|M_i|\ge5$ for all $i>0$, $V(\Omega) \subseteq V(M_0)$, $(M_0,\Omega)$ is $4$-connected, and $H_1'=S_3(M_0;M_1,\dots,M_a)$. Since $z_0$ has more than three neighbors in $\Omega$, $z_0$ must belong to $M_0$. It follows that $H_1=S_3(M_0;M_1,...,M_b)$ where $M_{a+1},...,M_b$ are $J_2,....,J_{3t}$, respectively. Note $H_1\backslash y_1y_2$ has an $\Omega$-cycle $z_1P_1y_1y_0y_2P_2z_2...z_{3t}z_1$, hence $M_0 \backslash y_1y_2$ also has an $\Omega$-cycle. 

If $M_0\backslash y_1y_2$ admits a planar drawing so that some facial cycle $F$ is an $\Omega$-cycle, let $G_0'$ be the 3-sum of $G_0$ and $M_0$. Then $G_0'$ is planar. Let $H_0'$ be obtained from $H_0$ by replacing edge $y_1y_2$ with path $F\backslash y_1y_2$. Then $H_0'$ is a subdivision of $W_k$ and the rim cycle of $H_0'$ is a facial cycle of $G_0'$. Moreover, $G=S_3^p(G_0'; H_2,...,H_h,M_1,...,M_b)$, which contradicts the maximality of $G_0$.

From Lemma~\ref{lem:Scycleedge}, $M_0\backslash y_1y_2$ has an $\Omega$-cycle $F$ and two crossing paths $Q_1,Q_2$ on $F$ with ends $q_1,q_3$ and $q_2,q_4$, respectively, such that among the four paths of $F$ divided by $q_1,q_2,q_3,q_4$, at least three of them contain edges of $\Omega$. For $i=1,2,3,4$, let $F_i=F[q_i, q_{i+1}]$, where $q_5=q_1$. We consider two cases. Suppose path $y_1y_0y_2$ is contained in some $F_i$, say $i=1$. Then one of $q_3,q_4$, say $q_3$, belongs to $\{z_3,z_4,...,z_{3t}\}$. It follows that $Q_1$ contains $z_0$  and thus $F_2\cup F_3$ contains the path $z_2z_3...z_{3t}z_1$. Without loss of generality, assume $q_3=z_{\lfloor3t/2\rfloor}$. Then the union of $Q_1, Q_2$, $F\backslash E(F_4)$, and $H_0$ contains $\theta_{t,t,t}$ at $q_2,q_3$, which settles this first case. Now we assume $y_0\in \{q_1,q_2,q_3,q_4\}$, and without loss of generality, $y_0=q_1$. We claim we may further assume that $F_2\cup F_3$ contains the path $z_2z_3...z_{3t}$. This is clear if $Q_2$ does not contain $z_0$. If $Q_2$ contains $z_0$ then $Q_1$ does not contain $z_0$, which implies either $F_1\cup F_2$ or $F_3\cup F_4$ contains path $z_2z_3...z_{3t}$. Let us assume the former, by symmetry. Then we can set $q_2=z_2$, which proves our claim. Therefore, either $Q_1\cup F_2$ or $Q_1\cup F_3$ is a long spoke and hence $G$ contains $\theta_{t,t,t}$. This completes our proof.
\end{proof}

Let $(G,w)$ be a weighted graph and suppose $G=S_d(G_0;G_1,\dots,G_k)$, where $d\in\{2,3\}$. Then we can define weights $w_0,w_1,\dots,w_k$. For each $i \geq 0$, if $e \in G_i$ does not belong to any summing triangle, then $w_i(e) = w(e)$. If $e \in G_i$ belongs to a summing triangle, then $w_i(e) = 1$. We say that $w_0,\dots,w_k$ are the {\it induced weights}.

\begin{proof}[Proof of Theorem \ref{thm:3need}]  
We show $f_{\ref{thm:3need}}(t)=f_{\ref{thm:ttt}}(t)$ satisfies the theorem. Let $(G,w)$ be 3-connected and $\theta_{t,t,t}$-free. Assume (b) does not hold. We first claim that there exists a 3-connected plane graph $G_0$ such that \\ 
\indent $\bullet$ if $C$ is the outer cycle of $G_0$ then either $|C|\ge 3t$ or $C$ contains at least three heavy edges, and \\ 
\indent $\bullet$ $(G,w)=S_3^p((G_0,w_0)$; $(G_1,w_1)$, $\dots$, $(G_k,w_k))$, where $G_i\in \mathcal{L}_{f_{\ref{thm:3need}}(t)}^3$ with $|G_i|\ge5$ for $i=1,...,k$. \\ 
This claim follows from Lemma \ref{thm:ttt} immediately if $G \notin \mathcal{L}_{f_{\ref{thm:3need}}(t)}^3$. So we assume $G \in \mathcal{L}_{f_{\ref{thm:3need}}(t)}^3$. Consider a cycle $Q$ containing as many heavy edges as possible. If there is a heavy edge $e$ not contained in $Q$ then $G$ has a $Q$-path $P$ containing $e$. It is easy to see that $Q\cup P$ either contains $\theta_{t,t,t}$ or contains a cycle that contains more heavy edges. Both cases are impossible, so $Q$ must contain all heavy edges. Let $\Omega$ be a circlet such that its edge set consists of all heavy edges, its vertex set consists of exactly vertices that are incident with at least one heavy edge, and such that $Q$ is an $\Omega$-cycle. Note $|E(\Omega)|\ge3$ and $|V(\Omega)|\ge4$ because (b) does not hold. By Lemma~\ref{lem:3sum}, $G$ has 3-connected minors $G_0,...,G_k$ such that $|G_i|\ge5$ for $i>0$, $V(\Omega) \subseteq V(G_0)$, $(G_0,\Omega)$ is $4$-connected, and $G=S_3(G_0;G_1,\dots,G_k)$. Note $G_0$ contains an $\Omega$-cycle since $G$ has an $\Omega$-cycle. By Theorem \ref{lem:Scycleedge}, $G_0$ admits a planar drawing with an $\Omega$-cycle $C$ as a facial cycle. Let $w_0,\dots,w_k$ be the induced weights. Then our claim holds with our choices of $(G_0,w_0)$, $(G_1,w_1)$, $\dots$,$(G_k,w_k)$, and $C$.

Let us choose $G_0$ satisfying the above claim with as many vertices as possible. If $G_i$ is 3-summed to $G_0$ over triangle $T$, then we assume no edge of $G_i$ is parallel to any edge of $T$ since we may put all these edges in $G_0$. We also assume each edge $e$ of $C$ has the maximum weight among all edges of $G_0$ that are parallel to $e$. By Lemma~\ref{lem:removeC}, $G_0$ contains no $C$-path of weight at least $2t$ and $w_0(e) < t$ for all edges $e$ of $G_0 \backslash E(C)$. Hence we conclude $(G_0,w_0) \in \mathcal P_t^3$. 

It remains to show that no $G_i$ ($i>0$) contains a heavy edge. Suppose to the contrary that some $G_i$ contains a heavy edge $e$. Let $T$ be the summing triangle of $G_i$. Then at most one end of $e$ is in $T$. By the maximality of $G_0$ and Lemma \ref{lem:tri-ext}, $G_i$ contains a minor $A\in\{A_1,A_2\}$. Note at least one vertex of $T$, say $v$, is not on $C$. Thus the 3-sum of $(G_0,w_0)$ and $(G_i,w_i)$ contains a minor $(G_0',w_0')$ obtained as follows: first we reduce $G_i\backslash E(T)$ to $A\backslash E(T)$, then we reduce $A\backslash E(T)$ to a triangle (by contracting two edges and deleting one or two edges) with vertex set $V(T)$ and such that $e$ is on the triangle and is incident with $v$. Then by applying Lemma \ref{lem:removeC} to $(G_0',w_0')$ we obtain a $\theta_{t,t,t}$. This contradiction completes our proof of the theorem. 
\end{proof}

\section{Proving the main theorem}

In this section we prove Theorem \ref{thm:big}. 
We divide the proof into two parts.

\begin{lemma}
There exists a function $f_{\ref{lem:ttt2}}(r,s)$ such that all weighted graphs in $\Phi(\mathcal L_{r,s},\mathcal P_r)$ are $\theta_{t,t,t}$-free, where $t=f_{\ref{lem:ttt2}}(r,s)$.
\label{lem:ttt2}
\end{lemma}

\begin{proof} We show $f_{\ref{lem:ttt2}}(r,s)=2qr$ satisfies the theorem, where $q=\max\{r,s\}-1$. Suppose there is a counterexample $(G,w)$. Then we choose one with $|G|$ minimum. Assume $(G,w)$ is formed by $k$-summing ($k=2,3,4$) weighted graphs $(G_1,w_1),...,(G_n,w_n)\in \mathcal L_{r,s}$ to $(G_0,w_0)\in \mathcal P_r$. Let $C$ be the outer cycle of $G_0$.

Suppose some $(G_{i_0},w_{i_0})$ is 4-summed to a rectangle $x_1x_2x_3x_4x_1$ of $G_0$, where $x_1x_2$ and $x_3x_4$ are edges of $C$. Recall that by the definition of a rectangle, this means no graph $(G_{i_1}, w_{i_1})$ can be 2-summed to an edge between $x_1$ and $x_2$ or $x_3$ and $x_4$ since there are no parallel edges between these vertices. We consider two cases. Assume first that $G$ has a 2-separation $(H,J)$ with $V(H\cap J) =\{x_j,x_{5-j}\}$ for $j=1$ or 2 and such that $C[x_j,x_{5-j}]\subseteq H$ and $C[x_{5-j},x_j]\subseteq J$. Define $(H^+,w_H)$ where $H^+$ is obtained by adding a new edge $e_H=x_jx_{5-j}$ to $H$, $w_H(e_H)$ is equal to the maximum weight of an $x_jx_{5-j}$-path in $J$, and $w_H(e)=w(e)$ for all edges $e$ of $H$. Also define $(J^+,w_J)$ analogously. Then $G_0$ can be expressed as a 2-sum of plane graphs $G_0^H$ and $G_0^J$ over $e_H$ and $e_J$ such that the outer cycles of $G_0^H$ and $G_0^J$ are $C[x_j,x_{5-j}]+e_H$ and $C[x_{5-j},x_j]+e_J$, respectively. Moreover, $(G_1,w_1),...,(G_n,w_n)$ can be divided into two groups such that the first group is summed to $G_0^H$ to obtain $(H^+,w_H)$ and the second group is summed to $G_0^J$ to obtain $(J^+,w_J)$. It follows that both $(H^+,w_H)$ and $(J^+,w_J)$ belong to $\Phi(\mathcal L_{r,s},\mathcal P_r)$. By the minimality of $G$, both $(H^+,w_H)$ and $(J^+,w_J)$ are $\theta_{t,t,t}$-free and thus, by Lemma \ref{lem:lemmab}, $(G,w)$ is also $\theta_{t,t,t}$-free. This is a contradiction and so the first case is settled.

Now in the second case, $G$ does not have a 2-separation as described in the previous paragraph. Then the length of $C$ must be 4 and $G_{i_0}$ must be the only graph summed to $G_0$ (so $n=1$). Therefore, $G_0$ consists of the 4-cycle $x_1x_2x_3x_4x_1$ and possibly more edges parallel to $x_1x_4$ or $x_2x_3$. Consequently, $G$ is obtained from $G_1\backslash \{x_1x_2,x_3x_4\}$ by adding parallel edges. Since all heavy edges of $G$ belong to $C$ and $x_1x_2$ and $x_3x_4$ are deleted after the sum, we deduce $G$ has at most two heavy edges. As a result, in every $\theta_{a,b,c}$ of $(G,w)$, at least one of its three independent paths cannot have any heavy edges. Let $t^*$ be the largest integer so that $(G,w)$ contains $\theta_{t^*,t^*,t^*}$. Then $t^*\le (r-1)(s-1)< f_{\ref{lem:ttt2}}(r,s)$. 

Now we assume that no $G_i$ is 4-summed to $G_0$. Suppose $x,y$ are distinct vertices of $G$ and $P_1,P_2,P_3$ are independent $xy$-paths of $G$. Let $p=\min\{w(P_1), w(P_2), w(P_3)\}$. We prove $p< 2qr$. 
If $P_j\subseteq G_i$ for some $j$ and $i>0$ then $p \le w_i(P_j) \le \max\{w_i(P):P$ is a path of $G_i\} \le (r-1)(s-1)< 2qr$. Henceforth we assume no $G_i$ contains any $P_j$. In particular, each $G_i-V(G_0)$ contains at most one of $x,y$. 

We modify $(G_0,w_0)$ and $P_1,P_2,P_3$ as follows. Let $P=P_1\cup P_2\cup P_3$. For each $i$ such that $G_i-V(G_0)$ contains neither $x$ nor $y$, note $G_i\cap P$ consists of zero, one, or two $G_i\cap G_0$-paths. If $Z$ is such a path with ends $z_1,z_2$, we change $w_0(z_1z_2)$ to $w_i(Z)$ and, in $P$, we replace path $Z$ by a single edge $z_1z_2$. If $G_i-V(G_0)$ contains $x$ or $y$, say $x$, then $V(G_i\cap G_0)$ consists of three vertices $z_1,z_2,z_3$, and we add a new vertex $x'$ and three new edges $x'z_1,x'z_2,x'z_3$ to $G_0$. In this case we define the weight of $x'z_j$ ($j=1,2,3$) to be $w_i(Z_j)$, where $Z_j$ is the $xz_j$-path contained in $G_i\cap P$. 
We also change $w_0(z_jz_{j'})$ to $w_i(Z_j)+w_i(Z_{j'})$. 
Let $(G_0',w_0')$ be the modified weighted graph. Let $P_1',P_2',P_3'$ be the three modified paths and $x',y'$ be their ends. Note $w_0'(P_j')=w(P_j)$ for $j=1,2,3$.

Note $G_0'$ is planar and let $C'$ be its outer cycle. We may assume that $P_2'$ is inside the region bounded by cycle $P_1'\cup P_3'$ and $C'$ is outside this region. Let $Q_1,Q_2$ be two disjoint paths from $C'$ to $P_1'\cup P_3'$. Then $P_1'\cup P_2'\cup P_3'\cup Q_1\cup Q_2$ contains a $C'$-path $Q'$ such that $P_2'\subseteq Q'$. Since the only possible vertices in $V(G_0') \backslash V(G_0)$ are $x',y'$ and each of them is surrounded by a triangle of $G_0$, we deduce $G_0$ has a $C'$-path $Q$ with $w_0'(Q)= w_0'(Q')$. Therefore, $p\le w_0'(P_2')\le w_0'(Q') = w_0'(Q) \le q||Q|| < 2qr$. \end{proof}

\begin{theorem}
There exists a function $s(t)$ such that every $2$-connected $\theta_{t,t,t}$-free weighted graph belongs to $\Phi(\mathcal L_{t,s(t)}, \mathcal P_t)$.
\label{thm:2connw}
\end{theorem}

\begin{proof}
We prove $s(t) =4t^2(3f_{\ref{thm:3need}}(t)+2)$ satisfies the theorem. Suppose there is a counterexample $(G,w)$. Then we choose one with $|G|$ minimum. If $|G|=2$, since $(G,w)$ is $\theta_{t,t,t}$-free, $G$ must have at most two heavy edges and thus $(G,w) \in \mathcal P_t\subseteq \Phi(\mathcal L_{t,s(t)}, \mathcal P_t)$. This contradicts the choice of $(G,w)$, so we assume $|G|\ge3$. We consider three cases based on Lemma~\ref{lem:2sep}.

\noindent \underline{Case (a) holds}: Let $(H,J)$ be a $2$-separation of $G$ with $V(H\cap J)=\{x,y\}$ such that neither $H$ nor $J$ has an $xy$-path of weight at least $t$. It is clear that $G$ has no heavy edges and, by Lemma~\ref{lem:lemmaa}, $G$ has no path of length at least $4t^2$. Hence $(G,w)\in\mathcal L_{t,s(t)}$. Since $G$ can be considered as a 2-sum of $G$ with a 2-cycle, and any weighted 2-cycle belongs to $\mathcal P_t$, it follows that $(G,w)\in \Phi(\mathcal L_{t,s(t)}, \mathcal P_t)$ as required.

\noindent \underline{Case (b) holds}: Let $(H,J)$ be a $2$-separation of $G$ with $V(H\cap J)=\{x,y\}$ such that $H$ and $J$ each have an $xy$-path of weight at least $t$. Denote by $(H^+,w_H)$ the graph formed from $H$ by adding an edge $e_H=xy$ with $w_H(e_H)$ equal to the weight of a heaviest $xy$-path in $J$ and $w_H(e)=w(e)$ for all other edges $e$. Define $(J^+,w_J)$ analogously. Now since $(G,w)$ is a minimal counterexample and, by Lemma~\ref{lem:lemmab}, both $(H^+,w_H)$ and $(J^+,w_J)$ are $\theta_{t,t,t}$-free, they both belong to $\Phi(\mathcal L_{t,s(t)}, \mathcal P_t)$. 

Let $(H_0,\alpha_0)\in \mathcal P_t$ be the base graph for constructing $(H^+, w_H)$ and let $C_H$ be the outer cycle of $H_0$. Let $(J_0,\beta_0)$ and $C_J$ be defined analogously. Since $e_H$ and $e_J$ are both heavy, $e_H \in C_H$ and $e_J \in C_J$. Let $(G_0,w_0)$ be the 2-sum of $(H_0,\alpha_0)$ and $(J_0,\beta_0)$ over $e_H$ and $e_J$, and let $C$ be the 2-sum of $C_H$ and $C_J$ over $e_H$ and $e_J$. Then $G_0$ is a plane graph with outer cycle $C$. In fact, $(G_0,w_0) \in \mathcal P_t$ because every $C$-path of $G_0$ is a $C_H$-path of $H_0$ or a $C_J$-path of $J_0$, and every heavy edge of $G_0$ is a heavy edge of $H_0$ or $J_0$. 

Let $\cal G$ be the set of weighted graphs that are summed to $(H_0,\alpha_0)$ or $(J_0,\beta_0)$ in forming $(H^+, w_H)$ and $(J^+, w_J)$. We claim that $(G,w)$ is formed by summing members of $\cal G$ to $(G_0,w_0)$. Since $e_H\in H^+$, $e_H$ is not contained in any summing 3- or 4-cycle of $H_0$. Moreover, every inner facial cycle of $H_0$ that does not contain $e_H$ remains an inner facial cycle of $G_0$. So summing edges and summing cycles of $H_0$ can still serve as a summing edge or cycle of $G_0$. Similarly, summing edges and summing cycles of $J_0$ can still serve as a summing edge or cycle of $G_0$. Therefore, the claim follows and thus $(G,w)\in \Phi(\mathcal L_{t,s(t)}, \mathcal P_t)$.

\noindent \underline{Case (c) holds}: Let $G=S_2(G_0;G_1,...,G_k)$ where $G_0,...,G_k$ satisfy Lemma \ref{lem:2sep}(c). Let $w_0,...,w_k$ be the induced weights. By Lemma~\ref{lem:lemmaa}, $(G_i,w_i)\in\mathcal L_{t,4t^2}$ for $i=1,...,k$. Moreover, heavy edges of $(G,w)$ are exactly heavy edges of $(G_0,w_0)$. First suppose $si(G_0)=K_3$. If no two heavy edges of $G_0$ are parallel, then $(G_0,w_0)\in\mathcal P_t$ and thus $(G,w)\in \Phi(\mathcal L_{t,4t^2}, \mathcal P_t)$. Assume $G_0$ has two parallel heavy edges $e,f$. Then they are the only two heavy edges since $(G,w)$ is $\theta_{t,t,t}$-free. Define $(G_0',w_0')$ where $G_0'$ consists of $e,f$ and a new edge $g$ parallel to $e,f$, and $w_0'(e)=w(e)$, $w_0'(f)=w(f)$, $w_0'(g)=1$. Let $(G_1',w_1')$ be obtained from $(G,w)$ by deleting $e,f$ and adding $g$ with weight 1. Then $(G,w)$ is the 2-sum of $(G_0',w_0')$ and $(G_1',w_1')$ over $g$. It is clear that $(G_0',w_0')\in\mathcal P_t$ and, by Lemma \ref{lem:Sdecompbd}, $(G_1', w_1')\in \mathcal L_{t,16t^2}$. Again we have $(G,w)\in \Phi(\mathcal L_{t,s(t)}, \mathcal P_t)$.

Second suppose $G_0$ is 3-connected. By Lemma \ref{lem:lemmab}, $(G_0,w_0)$ is $\theta_{t,t,t}$-free. We claim that $(G_0,w_0)\in \Phi(\mathcal L_{t,3f_{\ref{thm:3need}}(t)}, \mathcal P_t)$. By Theorem~\ref{thm:3need}, we assume $G_0 \in \mathcal L_{f_{\ref{thm:3need}}(t)}$ and either $(G_0,w_0)$ has at most two heavy edges or $(G_0,w_0)$ has exactly three heavy edges and these three form a triangle. Our claim is clear if $(G_0,w_0)$ has zero, one, two parallel, or three heavy edges: take the base graph to be a facial cycle of $G_0$ containing all of the heavy edges with an additional parallel edge added to each edge of the cycle. 
Suppose $(G_0,w_0)$ has two adjacent heavy edges $e=xy$ and $f=xz$ with $y\ne z$. Define $(G_0',w_0')$ where $G_0'$ is obtained from $e,f$ by adding three new edges $xy, yz, xz$, and $w_0'(e)=w_0(e)$, $w_0'(f)=w_0(f)$, $w_0'(xy)=w_0'(yz)=w_0'(xz)=1$. Let $(G_1',w_1')$ be obtained from $(G_0,w_0)$ by deleting $e,f$ and adding $xy,yz,xz$ of weight 1. Then $(G_0,w_0)$ is a 3-sum of $(G_0',w_0')$ and $(G_1',w_1')$. Moreover, $(G_0',w_0')\in\mathcal P_t$ and $(G_1',w_1')\in \mathcal L_{t,2f_{\ref{thm:3need}}(t)}$ (as $G_1'\backslash xz \cong G_0$), and thus our claim holds in this case. 
Finally, suppose  $(G_0,w_0)$ has two nonadjacent heavy edges $e=x_1x_4$ and $f=x_2x_3$. Define $(G_0',w_0')$ where $G_0'$ is obtained from $e,f$ by adding a 4-cycle $x_1x_2x_3x_4x_1$, and $w_0'(e)=w_0(e)$, $w_0'(f)=w_0(f)$, $w_0'(x_1x_2)=w_0'(x_2x_3)=w_0'(x_3x_4)=w_0'(x_4x_1)=1$. Let $(G_1',w_1')$ be obtained from $(G_0,w_0)$ by deleting $e,f$ and adding $x_1x_2,x_2x_3,x_3x_4,x_4x_1$ of weight 1. Then $(G_0,w_0)$ is a 4-sum of $(G_0',w_0')$ and $(G_1',w_1')$. Moreover, $(G_0',w_0')\in\mathcal P_t$ and $(G_1',w_1')\in \mathcal L_{t,3f_{\ref{thm:3need}}(t)}$, and thus our claim is proved. 

By this claim, $(G_0,w_0)$ is formed by 2-, 3-, and 4-summing $(H_1,\alpha_1)$, ..., $(H_n,\alpha_n)\in\mathcal L_{t,3f_{\ref{thm:3need}}(t)}$ to $(H_0,\alpha_0)\in\mathcal P_t$. Now weighted graphs $(G_1,w_1),...,(G_k,w_k)$ can be divided into groups $\mathcal H_0, ..., \mathcal H_n$ such that $(G_i,w_i)$ belongs to $\mathcal H_j$ if $G_i$ is 2-summed to $H_j$. For each $j>0$, let $(H_j^*,\alpha_j^*)$ be obtained by 2-summing all weighted graphs in $\mathcal H_j$ to $(H_j,\alpha_j)$. Then $(G,w)$ is obtained by 2-, 3-, 4-summing members of $\mathcal H_0\cup \{ (H_1^*,\alpha_1^*), ..., (H_n^*,\alpha_n^*)\}$ to $(H_0,\alpha_0)$. It remains to show $\mathcal H_0\cup\{(H_1^*,\alpha_1^*), ..., (H_n^*,\alpha_n^*)\}\subseteq \mathcal L_{t,s(t)}$. Since each $(G_i,w_i)$ has no $x_iy_i$-path of weight $\ge t$, we must have $\mathcal H_0\subseteq \mathcal L_{t,4t^2}$ by Lemma \ref{lem:lemmaa}. Moreover, by Lemma \ref{lem:Sdecompbd}, $\ell(H_j^*)\le (3f_{\ref{thm:3need}}(t)+2)(4t^2)$ and thus $\{ (H_1^*,\alpha_1^*), ..., (H_n^*,\alpha_n^*)\} \subseteq \mathcal L_{t,s(t)}$. Therefore, $(G,w)\in \Phi(\mathcal L_{t,s(t)}, \mathcal P_t)$, which completes our proof.
\end{proof}

\begin{proof}[Proof of Theorem \ref{thm:big}] The theorem is proved by Lemma \ref{lem:ttt2} and Theorem 
\ref{thm:2connw}. \end{proof}


\end{document}